\documentclass[12pt, a4paper]{amsart}
\usepackage{amssymb,amsmath,amsfonts,amsthm}
\usepackage{yhmath}
\usepackage{url}
\usepackage{preamble}
\usepackage{tikz-cd}
\usepackage{color}

\usepackage{todonotes}

\newcommand{\tf}{{\tilde f}}
\newcommand{\tth}{{\tilde h}}
\newcommand{\tg}{{\tilde g}}
\newcommand{\tW}{{\widetilde W}}
\newcommand{\thF}{{\widetilde {\mathcal {F}}}}
\newcommand{\tP}{{\mathcal P}}
\newcommand{\tA}{{\mathcal A}}

\newcommand{\tpi}{{\tilde \pi}}
\newcommand{\tS}{{\tilde S}}
\newcommand{\Diff}{\mathrm{Diff}^1}
\newcommand{\td}{\mathrm{d}}
\newtheorem*{maintheorem}{Main Theorem}

\def\vep{\varepsilon}
\def\cF{{\mathcal F}}
\def\supp{\operatorname{supp}}
\def\diam{\operatorname{diam}}

\title
[$C^1$ entropy formula part 1]
{
    Ledrappier-Young entropy formula for $C^1$ diffeomorphisms with dominated splitting \\
    Part 1: unstable entropy formula and invariance principle
}
\begin{document}
\begin{abstract}
We study the unstable entropy of $C^1$ diffeomorphisms with dominated splittings. Our main result shows that when the zero Lyapunov exponent has multiplicity one, the center direction contributes no entropy, and the unstable entropy coincides with the metric entropy. This extends the celebrated work of Ledrappier–Young \cite{LEDRAPPIER_YOUNG_A} for $C^2$ diffeomorphisms to the $C^1$ setting under these assumptions. In particular, our results apply to $C^1$ diffeomorphisms away from homoclinic tangencies due to \cite{LVY13}.

As consequences, we obtain several applications at $C^1$ regularity. The Avila–Viana invariance principle \cite{crovisier2023invarianceprinciplenoncompactcenter, TY} holds when the center is one-dimensional. Results on measures of maximal entropy due to Hertz–Hertz–Tahzibi–Ures \cite{HHTU}, Tahzibi–Yang \cite{TY}, and Ures–Viana–Yang–Yang \cite{UVYY2025, UVYY2024} also remain valid for $C^1$ diffeomorphisms.

\end{abstract}
\maketitle

\tableofcontents

\section{Introduction and main results}
\subsection{Introduction}
The notion of metric entropy, introduced by Kolmogorov and Sinai in the 1950s, provides a way to quantify the complexity of chaotic dynamical systems. In the 1960s, Oseledets’ theorem \cite{Oseledets} introduced Lyapunov exponents, another fundamental tool in smooth ergodic theory. These two concepts are deeply connected: the inequality of Margulis and Ruelle, see \cite{1978Ruelle-inequality}, shows that the metric entropy of an invariant measure for diffeomorphisms is upper bounded by the sum of the positive Lyapunov exponents, counting multiplicity; moreover, for $C^{1+\mathrm{H\"{o}lder}}$ volume-preserving diffeomorphisms, the Pesin's entropy formula \cite{Pesin-entropyformula_1977} establishes equality, namely the entropy of the volume measure is equal to the sum of its positive Lyapunov exponents.

In a series of remarkable works \cite{Ledrappier_Strelcyn_1982}, \cite{Ledrappier1984}, and \cite{LEDRAPPIER_YOUNG_A} in the 1980s, it was proven that for $C^2$ diffeomorphisms, the Pesin entropy formula holds if and only if the measure admits absolutely continuous conditional measures along Pesin's unstable manifolds. To achieve this, the authors of these papers introduced the concept of unstable entropy, i.e., entropy along Pesin unstable manifolds, and the main step in \cite{LEDRAPPIER_YOUNG_A} is to prove that the unstable entropy is equal to the metric entropy for any $C^2$ diffeomorphisms. 

The assumption of $C^2$ regularity is essential in the classical results described above. The key reason is that the holonomy maps induced by unstable foliations within center–unstable sets are Lipschitz.
This Lipschitz property also holds for $C^{1+\mathrm{H\"{o}lder}}$ diffeomorphisms, as shown in recent works of Brown in \cite{AaronBrown2022} and Saghin in \cite{SaghinRadu2025}, thereby extending the Ledrappier-Young theory to the $C^{1+\mathrm{H\"{o}lder}}$ setting.

In contrast, $C^1$ diffeomorphisms may fail to admit stable manifolds \cite{Pugh1984, BCS2013}. Nevertheless, several results \cite{Gan2002, Tahzibi_2002, SunTianC1pesinformula, ABC} suggest the possibility of a parallel theory for ``$C^1$ + domination,'' analogous to the $C^{1+\mathrm{H\"{o}lder}}$ or $C^2$ cases.
Along this direction, a central and challenging problem is to establish the Ledrappier–Young entropy formula for $C^1$ diffeomorphisms with dominated splitting. In this setting, the unstable holonomy maps inside center–unstable sets are generally not Lipschitz (rather they are H\"older). In this paper, we provide a partial positive answer to this problem in the case where the center bundle is one-dimensional.

We now briefly describe the framework. Let $M$ be a compact $C^{\infty}$ Riemannian manifold without boundary, and let $f \in \Diff(M)$ be a $C^1$ diffeomorphism. Denote by $m$ an $f$-invariant ergodic Borel probability measure on $M$. 
The main result of this paper is that if the Oseledets splitting $ E^u \oplus  E^c \oplus E^s$ corresponding to the positive, zero and negative exponents bundles is a dominated splitting on the support of $m$, and if the zero exponent has multiplicity one, then the center contributes no entropy.
\begin{maintheorem}[Section 1.2, \Cref{main thm}]
Let $f\in \Diff(M)$ preserve an ergodic Borel probability $m$. If the Oseledets splitting $ E^u \oplus  E^c \oplus  E^s $ is dominated on the support of $m$ and $\dim  E^c \leq1$, then we have $h^u_m(f)=h_m(f)$, where $h_m(f)$ denotes the metric entropy and $h_m^u(f)$ the unstable entropy.
\end{maintheorem}
See \Cref{section 1.2} for precise definitions. In particular, the assumptions of the Main Theorem hold for any ergodic measure of diffeomorphisms away from homoclinic tangencies. See \Cref{section 1.3} for details. 

We now turn to the invariance principle, first established by Ledrappier in \cite{Ledrappier_86} for linear cocycles and later extended to the general setting by Avila and Viana in \cite{AV}. It is a powerful tool in smooth ergodic theory with numerous important applications. 
Although not explicitly stated in Avila-Viana's work, the Lipschitz regularity for the unstable holonomy maps inside center-stable manifolds (as in the Ledrappier-Young's theory) is in fact a necessary assumption.  

As an application of our Main Theorem, we show that the Avila–Viana invariance principle continues to hold in the $C^1$ category when the center is one-dimensional.
Indeed, Tahzibi and Yang \cite[Corollary 2.1]{TY} proved that a measure is $u$-invariance if and only if its entropy along unstable foliation coincides with the entropy of the quotient measure along the unstable foliation of the quotient space.
From this perspective, the invariance principle follows from our Main Theorem, which ensures that zero Lyapunov exponents contribute no entropy. Consequently, several recent works that rely on the invariance principle also remain valid for $C^1$ diffeomorphisms, including \cite{HHTU}, \cite{TY}, \cite{UVY}, \cite{UVYY2024}, and \cite{UVYY2025}; see \Cref{section6} for details.

Finally,we emphasize that the Lipschitz regularity of holonomy maps is a fundamental assumption for many remarkable results in smooth ergodic theory, such as \cite{BW2010}. 
In future works, we aim to relax the Lipschitz regularity to Hölder regularity in some important cases. In particular, in Part 2 of this paper, 
we will establish a $C^1$ version of the Ledrappier-Young entropy formula from their seminal work \cite{LEDRAPPIER_YOUNG_B}.

\smallskip 
\textbf{Strategy of proof:}
We now outline the strategy used to overcome the lack of regularity.

In the $C^1$ setting, the Lyapunov charts in \cite{LEDRAPPIER_YOUNG_A} are no longer available, since the sub-exponential variation on the size of the charts under iteration depends heavily on the $C^{1+\mathrm{H\"{o}lder}}$ regularity. 
Instead, the presence of a dominated splitting allows us to construct fake foliations in local chart neighborhoods. Moreover, the results of Abdenur, Bonatti and Crovisier in \cite{ABC} enables the construction of fake foliation charts, see \Cref{fake foliation charts def}, to replace the Lyapunov charts. It is worth noting that the size of the fake foliation charts are non-uniform, scaling as the reciprocal of the hyperbolicity function of the base point (\Cref{ABC ergodic lemma}). Although these charts lack the one-step length estimates as for Lyapunov charts in \cite[Section 2.1]{LEDRAPPIER_YOUNG_A}, they still provide sufficient asymptotic length control (\Cref{local argument}), which is sufficient for our purpose.

To address the absence of Lipschitz regularity for unstable holonomies, we first show that the fake unstable holonomies inside the fake center-unstable leaves are bi-H\"older continuous, with H\"older exponent arbitrarily close to one; see \Cref{uniform Hölder}. While this fact seems to be well-known, we are not aware of an explicit reference for our precise setting. Next, we carefully modify the convergence argument in \cite[Section 5]{LEDRAPPIER_YOUNG_A}. Instead of centered balls, we utilize the uncentered maximal function, where the balls are allowed to shift only along $E^c$. The one-dimensionality of $E^c$ is crucial here, as it enables us to establish the required convergence arguments for uncentered balls (\Cref{section 4}, \Cref{section 5.3}), thereby recovering the rest of the proof in \cite[Section 5]{LEDRAPPIER_YOUNG_A}.

\smallskip
\textbf{Organization:}
Section 2 contains the construction of the fake foliation charts and provides the necessary geometric preparations, including the H\"older continuity of fake unstable holonomies inside fake center–unstable leaves, as well as properties of unstable manifolds and center–unstable sets.
Section 3 introduces the notion of unstable entropy for $C^1$ diffeomorphisms with dominated splitting, in parallel with \cite{LEDRAPPIER_YOUNG_B}, and establishes the quotient structure of two special partitions. In Section 4 we present the uncentered maximal function, where the one-dimensional assumption plays a crucial role. In Section 5 we complete the proof of the Main Theorem. Finally, in Section \ref{section6} we state and prove the $C^1$ invariance principle, namely \Cref{thm invprin}.  

\subsection{Main results}\label{section 1.2}
In this section, we provide the technical definitions and state the main results of this paper.
As before, $f$ is a $C^1$ diffeomorphism on a compact Riemannian manifold $M$ without boundary, and $m$ is an $f$-invariant ergodic Borel probability measure.

A point $x\in M$ is said to be a regular point if there exist numbers $\lambda_1(x)>\lambda_2(x)>\cdots >\lambda_{r(x)}(x)$ and a decomposition of the tangent space $$T_xM=E_1(x)\oplus E_2(x)\oplus \cdots \oplus E_{r(x)}(x)$$ such that for each $v\neq 0\in E_i(x)$, we have $$\lim_{n\rightarrow \pm\infty} \frac{1}{n}\log | Df^n_x v|=\lambda_i(x),$$ and $$\lim_{n\rightarrow \pm \infty}\frac{1}{n}\log |\mathrm{Jac}(Df^n_x)|=\sum^{r(x)}_{i=1}\lambda_i(x)\dim E_i(x).$$
By the Oseledets theorem \cite{Oseledets}, the set of regular points, denoted as $\Gamma$, is a set of full $m$ measure and $\lambda_i(x),r(x)$ and $\dim E_i(x)$ are all $f$-invariant measurable functions. When $m$ is ergodic, for $m$-a.e.\ $x\in M$,
$\lambda_i(x)=\lambda_i$ are called the Lyapunov exponents of $f$, and $\dim E_i(x)=\dim E_i$ is called the multiplicity of $\lambda_i$. We denote the smallest positive exponent by $\lambda^u$ and the largest negative exponent by $\lambda^s$.

We recall some basic properties of dominated splittings (see \cite{ABC}). A $Df$-invariant splitting $T_\Lambda M=F\oplus E$ of the tangent bundle over an $f$-invariant set $\Lambda$ is dominated if there exists $N>0$ such that given any $x\in \Lambda$ and any unit vectors $v\in E(x),u\in F(x)$, we have $$| Df_x^N(v)|\leq\frac{1}{2}| Df_x^N(u)|.$$
More generally, a $Df$-invariant splitting $T_\Lambda M=E_1\oplus\cdots\oplus E_t$ is dominated if given any $\ell\in\{ 1,\cdots,t-1\}$ the splitting $(E_1\oplus\cdots\oplus E_\ell)\oplus(E_{\ell+1}\oplus\cdots\oplus E_t)$ is dominated. Recall that if an $f$ invariant set $\Lambda$ admits a dominated splitting $T_\Lambda M=E_1\oplus\cdots\oplus E_t$, then the splitting depends continuously on $x\in \Lambda$ and extends uniquely to a dominated splitting over the closure of $\Lambda$. Moreover, the angles between different bundles are uniformly bounded away from zero.

Define $ E^s (x)=\bigoplus_{\lambda_i<0} E_i(x)$, $ E^c (x)=\bigoplus_{\lambda_i=0} E_i(x)$ and $ E^u (x)=\bigoplus_{\lambda_i>0} E_i(x)$ for any $x\in\Gamma$. We say that the splitting $ E^u \oplus  E^c \oplus  E^s $ is dominated on the support of $m$, if $ E^u \oplus  E^c \oplus  E^s $ is a dominated splitting on the invariant set $\Gamma \cap \mathrm{supp}(m)$. In this case, the splitting can be uniquely extended to a dominated splitting on $\mathrm{supp}(m)$, since $\Gamma \cap \mathrm{supp}(m)$ is dense inside $\mathrm{supp}(m)$.

For any $x\in \Gamma\cap \mathrm{supp}(m)$ we define the (global) stable and unstable manifolds as $$W^s(x)=\left\{ y\in M: \limsup_{n\rightarrow +\infty}\frac{1}{n}\log d(f^nx,f^ny)<0 \right\}$$
and
$$W^u(x)=\left\{ y\in M: \limsup_{n\rightarrow +\infty}\frac{1}{n}\log d(f^{-n}x,f^{-n}y)<0 \right\},$$
where $d$ is the Riemannian metric on $M$. When $ E^u \oplus  E^c \oplus  E^s $ is dominated on the support of $m$, Abdenur, Bonatti, and Crovisier proved in \cite[Section 8]{ABC} that $W^u(x)$  (resp. $W^s(x)$) is an injectively immersed $C^1$ submanifold of dimension $\dim  E^u (x)$ (resp. $\dim  E^s $).

When $ E^u \oplus  E^c \oplus  E^s $ is dominated on the support of $m$, we define the $u$-entropy of $f$ following the ideas of \cite{LEDRAPPIER_YOUNG_A}, \cite{wuweisheng-u-entropy} and \cite{YANG_expanding_entropy}. We will use the concepts of measurable partitions, disintegration and entropy, for which the reader can find more details in \cite{Rokhlin_1967}. Similar to \cite{LEDRAPPIER_YOUNG_A}, we define the $u$-entropy through special partitions:

\begin{definition}[Subordinate partition]
Let $\xi$ be a measurable partition of $M$. We say $\xi$ is subordinate to $W^u$, if for $m$-a.e.\ $x$ we have:\\
a) $\xi(x)\subseteq  W^u(x) $ and $\xi(x)$ contains an open neighborhood of $x$ within $W^u(x)$.\\
b) $\xi$ is increasing, meaning that $f\xi \prec \xi$.\\
c) $\bigvee_{n=0}^\infty f^{-n}\xi$ is the partition into points. 
\end{definition}
The $u$-entropy, denoted by $h^u_m(f)$, is defined as the conditional entropy $h^u_m(f) :=  H(\xi|f\xi)$ for any $\xi$ subordinate to $W^u$. We will see that the $u$-entropy does not depend on the choice of $\xi$. More details will be given in \Cref{u-entropy}.

We now state our main result in this paper:
\begin{theorem}[Main Theorem]\label{main thm}
Let $f\in \Diff(M)$ preserve an ergodic Borel probability $m$. If the splitting $ E^u \oplus  E^c \oplus  E^s $ is dominated on the support of $m$ and $\dim  E^c \leq1$, then $h^u_m(f)=h_m(f)$, where $h_m(f)$ denotes the metric entropy of $f$.
\end{theorem}
\begin{remark}
We only need to prove the theorem for $\dim  E^c =1$. In the case $\dim  E^c =0$, i.e., $m$ is an ergodic hyperbolic measure, the result is proven by Ures, Viana, Yang and Yang in \cite[Proposition A.4]{UVYY2024}. Although their result was proven for partially hyperbolic diffeomorphisms, the proof is still valid under our setting with minor adjustments and is indeed easier than the case $\dim  E^c =1$.
\end{remark}

\subsection{Diffeomorphisms away from tangencies}\label{section 1.3}
We say that $f\in\mathrm{Diff}^1(M)$ is away from homoclinic tangencies if there exists a $C^1$ neighborhood $\mathcal{U}$ of $f$ such that for any $g\in\mathcal{U}$ and any hyperbolic periodic orbit $P$ of $g$, $W^s(P)$ is transverse to $W^u(P)$.

The following is a simple consequence of \cite[Proposition 3.4]{LVY13}.

\begin{proposition} Assume that $f\in\mathrm{Diff}^1(M)$ is away from homoclinic tangencies and $m$ is an ergodic measure of $f$. Then there exists a dominated splitting $T_{\mathrm{supp}(m)}M=E^u\oplus E^c\oplus E^s$, such that $\dim E^c\leq 1$, the Lyapunov exponents in $E^u$ are positive and the Lyapunov exponents in $E^s$ are negative.
\end{proposition}

Hence the zero exponent has multiplicity at most one. 
By the Main Theorem, we have $h^u_m(f)=h_m(f)$. This is summarized in the following theorem.
\begin{theorem}\label{away from tangent thm}
    If $f\in\mathrm{Diff}^1(M)$ is away from homoclinic tangencies and $m$ is an ergodic measure of $f$, then $h^u_m(f)=h_m(f)$. 
\end{theorem}

\subsection{The non-ergodic case}
If the measure $m$ is not an ergodic measure, we can use the Ergodic Decomposition Theorem to reduce to the ergodic case. Such an argument was used in \cite[Section 6.2]{LEDRAPPIER_YOUNG_A} to reduce the non-ergodic case to the ergodic case in the $C^2$ setting. The same technique applies to our setting, and we omit the details here.

However, one should note that by the Oseledets theorem, there exists a splitting of positive, zero and negative exponents bundles $E^u(x)\oplus E^c(x)\oplus E^s(x)$ for $x\in\Gamma$ as before, where $\Gamma$ is the regular point set, but the dimension of each bundle need not be constant for $m$-a.e.$\ x\in M$ if $m$ is not ergodic. If this splitting is dominated on the support of $m$ and $\dim  E^c (x)\leq 1$ for $m$-a.e.\ $x\in M$, then for $m$-a.e.\ $x\in M$, the ergodic component $m_x$ of $m$ satisfies the assumptions of \Cref{main thm}, i.e., the splitting $ E^u \oplus  E^c \oplus  E^s $ is dominated on the support of $m_x$ and $\dim  E^c (y)\leq 1$ for $m_x-$a.e.\ $y\in M$. In this case, we define the $u$-entropy of $m$ as the integral of the $u$-entropy of its ergodic component $m_x$, i.e., $$h^u_m(f) := \int h^u_{m_x}(f)\,\td m(x).$$  Then we obtain a non-ergodic version of the main Theorem.
\begin{theorem}
	Let $f\in \Diff(M)$ preserve a Borel probability $m$. If the splitting $ E^u \oplus  E^c \oplus  E^s $ is dominated on the support of $m$ and $\dim  E^c (x)\leq1$ for $m$-a.e.\ $x\in M$, then we have $h^u_m(f)=h_m(f)$.
\end{theorem}

Proceeding as above, one also obtains a non-ergodic version of \Cref{away from tangent thm}.
\begin{theorem}
	Assume that $f\in\mathrm{Diff}^1(M)$ is away from homoclinic tangencies and $m$ is an invariant measure of $f$, then $h^u_m(f)=h_m(f)$.
\end{theorem}

{Further applications, in particular, the $C^1$ invariance principle, are postponed to Section \ref{section6}.}

\section{The geometric structure}\label{section 2}
In this section which is parallel to \cite[Section 2]{LEDRAPPIER_YOUNG_A}, we introduce the fake foliations in \Cref{section 2.1} and prove the uniform H\"older property of the fake unstable holonomy maps inside the fake center unstable leaves in \Cref{section 2.2}. In \Cref{section 2.3}, we consider the unstable manifolds and the center unstable sets in the fake foliation charts, and prove properties analogous to \cite[Section 2.2]{LEDRAPPIER_YOUNG_A} and \cite[Section 2.3]{LEDRAPPIER_YOUNG_A} under our setting. In \Cref{section 2.4} we construct the adapted partition, similar to \cite[Section 2.4]{LEDRAPPIER_YOUNG_A}.

{Compared to \cite[Section 2]{LEDRAPPIER_YOUNG_A}, the main difficulty caused by the lack of $C^{1+\mathrm{H\"{o}lder}}$ regularity is that we can not obtain the sub-exponential variation for the size of the Lyapunov charts. 
	To solve this issue, we will construct charts using {fake foliations} (Definition \ref{fake foliation charts def}) to replace the Lyapunov charts in \cite{LEDRAPPIER_YOUNG_A}. We then use the local argument (see \Cref{local argument}) to obtain asymptotic length control and establish the properties parallel to \cite[Section 2]{LEDRAPPIER_YOUNG_A}.  }

\subsection{Fake foliation}\label{section 2.1}
In this section, we consider fake foliations arising from dominated splittings. More details on fake foliations can be found in \cite{BW2010} and \cite{LVY13}.

Throughout this paper, we assume that $f\in \Diff(M)$ preserves an ergodic Borel probability $m$, {and the Oseledets splitting extends to a continuous dominated splitting $ E^u \oplus  E^c \oplus  E^s $ on the support of $m$, where $E^u$ (resp. $E^c, E^s$) corresponds to the positive (resp. zero, negative) Lyapunov exponents of $m$.} By the compactness of $\supp(m),$ there exists a constant $\theta>0$ such that for any $x\in \mathrm{supp}(m)$ and all $i\neq j\in\{u,c,s\}$, we have $|\sin\angle(E^i(x),E^{j}(x))|>\theta$. In other words, the angles between different bundles are uniformly bounded away from zero.

The following lemma follows from \cite[Section 8.1]{ABC}.
\begin{lemma}[{\cite[Section 8.1]{ABC}}]\label{lemmaA(x)}
    Given any $\epsilon, \eta>0$, there exist a full measure set $\Gamma'\subseteq \mathrm{supp}(m)\cap \Gamma$, two measurable functions $A':\Gamma'\rightarrow[1,\infty)$ and $C:\Gamma'\rightarrow(0,\infty)$, and a constant $N\in\mathbb N$ such that given any $x\in\Gamma'$, one has 	
	\begin{equation*}
		C(x)^{-1}\mathrm{e}^{-\eta n} \leq A'(f^{\pm n}x)\leq C(x)\mathrm{e}^{\eta n},\ \forall n\in \mathbb N,
	\end{equation*}
	and for any $k\in \mathbb N $ we have:
	\begin{equation*}
		\prod_{l=0}^{k-1}\| Df^N|_{ E^s (f^{lN}(x))}\| \leq A'(x)\cdot \mathrm{e}^{kN(\lambda^s+\epsilon)} ,
	\end{equation*}
	\begin{equation*}
		\prod_{l=0}^{k-1}\| Df^N|_{ E^c (f^{lN}(x))}\| \leq A'(x)\cdot \mathrm{e}^{kN\epsilon} ,
	\end{equation*}
	\begin{equation*}
		\prod_{l=0}^{k-1}\| Df^{-N}|_{ E^c (f^{-lN}(x))}\| \leq A'(x)\cdot \mathrm{e}^{kN\epsilon} ,
	\end{equation*}
	\begin{equation*}
		\prod_{l=0}^{k-1}\| Df^{-N}|_{ E^u (f^{-lN}(x))}\| \leq A'(x)\cdot \mathrm{e}^{kN(-\lambda^u+\epsilon)}. 
	\end{equation*}
\end{lemma}

The following lemma follows by taking $\eta=\frac{\epsilon}{2}$ in the previous lemma and $A(x)=\max_{n\in \mathbb Z} \left\{ A'(f^nx)\cdot \mathrm{e}^{-|n|\epsilon}\right\}$.
\begin{lemma}\label{ABC ergodic lemma}
	Given any $\epsilon>0$, there exist a measurable function $A:\Gamma'\rightarrow[1,\infty)$ and a constant $N\in\mathbb{N}$, such that $A(f^{\pm}x)\leq \mathrm{e}^{\epsilon}A(x)$ and for any $k\in \mathbb N $ we have:
	\begin{equation*}
		\prod_{l=0}^{k-1}\| Df^N|_{ E^s(f^{lN}(x))}\| \leq A(x)\cdot \mathrm{e}^{kN(\lambda^s+\epsilon)},
	\end{equation*}
	\begin{equation*}
		\prod_{l=0}^{k-1}\| Df^N|_{ E^c (f^{lN}(x))}\| \leq A(x)\cdot \mathrm{e}^{kN\epsilon},
	\end{equation*}
	\begin{equation*}
		\prod_{l=0}^{k-1}\| Df^{-N}|_{ E^c (f^{-lN}(x))}\| \leq A(x)\cdot \mathrm{e}^{kN\epsilon},
	\end{equation*}
	and
	\begin{equation*}
		\prod_{l=0}^{k-1}\| Df^{-N}|_{ E^u (f^{-lN}(x))}\| \leq A(x)\cdot \mathrm{e}^{kN(-\lambda^u+\epsilon)}. 
	\end{equation*}
\end{lemma}
Fix a constant 
\begin{equation}\label{e.cf}
	C_f\geq 100\max\{\| Df\|,\mathrm{e}^{\lambda^s+100\epsilon},\mathrm{e}^{-\lambda^u+100\epsilon}\}.
\end{equation}
There exists a constant $\Delta>0$ such that for any $x\in M$, $\exp_x:T_xM\to M$ maps $\{v\in T_xM:|v|\leq \Delta\}$  diffeomorphically to its image, with $\sup_x\|D\exp_x^{\pm}\|\leq 2$. We define $\tf_x=\exp^{-1}_{fx}\circ f\circ \exp_x$ whenever it makes sense.

Choose a $C^\infty$ bump function $\rho:\mathbb R\rightarrow\mathbb R$ such that
\begin{equation*}
	\rho(x)=\left\{
	\begin{aligned}
		1 & , &|x|&<1\\
		0 & , & |x|&>2
	\end{aligned}
	\right.
\end{equation*}
and satisfies $\rho \in [0,1]$ for all $x\in \mathbb R$.
For any $r_0<0.1\Delta$, we can define $\tg_x:T_xM\rightarrow T_{fx}M$ for $x\in M$ as
\begin{equation*}
	\tg_x(v)=\left\{
	\begin{aligned}
		&\rho(|v|\cdot r_0^{-1})\tf_x(v)+(1-\rho(|v|\cdot r_0^{-1}))D\tf_x(0)v  , &|v|\leq 2r_0.\\
		&D\tf_x(0)v   , & |v|>2r_0.\\
	\end{aligned}
	\right.
\end{equation*}
Then we can calculate  
\begin{equation*}
	\begin{aligned}
		&\|D(\tg_x-D\tf_x(0))\|\\&\leq \rho(|v|\cdot r_0^{-1})\|D\tf_x(v)-D\tf_x(0)\|+\rho'(|v|\cdot r_0^{-1})r_0^{-1}\|\tf_x(v)-D\tf_x(0)\|\\
		&\leq \|D\tf_x(v)-D\tf_x(0)\|+C\cdot r_0^{-1}|v| \max_{|v'|\leq 2r_0}\|D\tf_x(v')-D\tf_x(0)\|\\
		&\leq 3C\cdot\max_{|v'|\leq 2r_0}\|D\tf_x(v')-D\tf_x(0)\|,
	\end{aligned}
\end{equation*}
where the constant $C>0$ can be chosen as the $C^1$ norm of $\rho$. By the uniform continuity of $Df$, for any $\epsilon_0>0$, there exists $r_0>0$ such that $D(\tg_x-D\tf_x(0))\leq \epsilon_0$. We use $L(g)$ to denote the Lipschitz constant of $g$; then we have $L(\tg_x-D\tf_x(0))\leq \epsilon_0$ and $\tg_x=\tf_x$ when $|v|\leq r_0$. Hence we can choose $\epsilon_0'>0$ sufficiently small such that for any $\epsilon_0\leq\epsilon_0'$ and $\tg_x$, $N$ defined as above, we have 
\begin{equation}\label{local argument 1}
	\mathrm{e}^{-\epsilon N}\leq \frac{|D\tf^{\pm N}_x(0)v|}{|D\tg^{\pm N}_x(y)v|}\leq \mathrm{e}^{\epsilon N}, \ \forall x\in M,\ v ,y\in T_xM .   
\end{equation}

Unlike \cite{LEDRAPPIER_YOUNG_A}, in this article we will not use the Lyapunov norms. Instead, we use the box norm $|\cdot|'$ on each $T_xM,x\in\Gamma'$ w.r.t.\ the splitting $T_xM= E^u _x\oplus E_x^c\oplus  E^s_x$, that is, $|v|'=\max\{|v_u|,|v_c|,|v_s|\}$ where $v=v_u+v_c+v_s$ is the decomposition as above, and $|\cdot|$ is the norm induced by the Riemann metric on $M$.

We then denote by $R^i_x(\rho)$ the ball centered at $0$ with radius $\rho$ under the box norm in $E^{i}_x$, where $i\in\{u,c,s,cu,cs,us\}$. Similarly, $R_x(\rho)$ (without any superscript) is the ball centered at $0$ with radius $\rho$ under the box norm in $T_xM$. Since the angles between different bundles are uniformly bounded away from zero, we have a constant $K>0$ such that $K^{-1}|\cdot|\leq |\cdot|'\leq K|\cdot|$. 

Henceforth we slightly abuse notation by treating $u,c,s,cu,cs,us$ as the dimension of their corresponding bundles (i.e., $u = \dim E^u$ and so on). 
By continuity, there exists a constant $C_0>K$ sufficiently large such that for any $x$ in the full measure set $\Gamma'$, and any $i$-dimensional linear subspace $V^i(x)$ as the graph of a linear function $G:E^i_x\rightarrow E^{j}_x$ with slope $\leq \frac{1}{C_0}$, $i,j\in\{u,c,s,cu,cs,us\}$ and $T_xM=E^i_x\oplus E^{j}_x$, we have
\begin{equation}\label{local argument 2}
	\mathrm{e}^{-\epsilon N}\leq\frac{\|D\tf_x^{\pm N}(0)|_{V^i(x)}\|}{\|D\tf_x^{\pm N}(0)|_{E^i(x)}\|}\leq \mathrm{e}^{\epsilon N},
\end{equation}
where $N$ is given by Lemma \ref{lemmaA(x)}.
We note that $C_0$ is a constant that depends only on $(f,m)$, $N$ and $\epsilon$ but not on $\epsilon_0$, $r_0$ or $\tg_\cdot$.

Now we consider the splitting $E^{cu}\oplus  E^s $. Since the splitting is dominated, there exists a large constant $C_1>0$ and a constant $\epsilon_0''\leq\epsilon'_0$, such that for any $\epsilon_0\leq\epsilon_0''$, and the $r_0$ and $\tg_\cdot$ given by $\epsilon_0$ as above, we have for any $x\in \Gamma'$, if $\varphi:E^{cu}_x\rightarrow  E^s_x$ is a smooth function with slope $\leq\frac{1}{C_1}$, then for any $n\geq 0$, we have $\tg^n_x(\mathrm{graph}(\varphi))\subseteq T_{f^nx}M$ is the graph of a function $\varphi_n:E^{cu}_{f^nx}\rightarrow  E^s_{f^nx}$ with slope $\leq\frac{1}{C_0}$. Similarly, if $\varphi:E^{s}_x\rightarrow E^{cu}_x$ is a smooth function with slope $\leq\frac{1}{C_1}$, then for any $n\leq 0$, we have $\tg^n_x(\mathrm{graph}(\varphi))\subseteq T_{f^nx}M$ is the graph of a function $\varphi_n:E^{s}_{f^nx}\rightarrow E^{cu}_{f^nx}$ with slope $\leq\frac{1}{C_0}$.
By enlarging $C_1$, we may further assume that the same conclusion holds for the dominated splitting $ E^u \oplus E^{cs}$. Keep in mind that the slopes here are all considered under the box norm.

We fixed a constant $\gamma_0<\frac{\pi}{10^6\cdot C_0^{100}}$. Following the standard graph transformation methods (see for instance \cite{Katok_book}, \cite{HPS}, \cite{BW2010} and \cite{LVY13}), we can choose $\epsilon_0>0$ and consequently $r_0>0$ small enough, such that the following lemma hold.

{
	\begin{lemma}[Fake foliations]\label{fake foliations}
		For each $x\in\Gamma'$ and $i\in\{u,s,cs,cu\}$, there exist unique global fake foliations $\thF^{i}_x$ on $T_xM$ with $C^1$ leaves, such that for any $y \in T_xM$:
		\begin{enumerate}
			\item[a)] The unique leaf containing $y$, denoted by $\tW^{i}_x(y)$, is the graph of a $C^1$ function $\varphi:E^{i}_x\rightarrow E^{\tau(i)}_x$ with $|D\varphi|\leq\frac{\gamma_0}{C_1}$. Here when $i = u,s,cu,cs$, $\tau(i) = cs,cu,s,u$ respectively. 
			\item[b)] The foliations are invariant under $\tg_{\cdot}$ in the sense that 
			$$
			g_x\left(\thF^{i}_x(y)\right) = \thF^{i}_{f(x)}(g_x(y)).
			$$
			\item[c)] $\thF^{u}_x$ subfoliates $\thF^{cu}_x$, and $\thF^{s}_x$ subfoliates $\thF^{cs}_x$.
		\end{enumerate}
	\end{lemma}
}
{Note that Lemma \ref{fake foliations} does not immediately give the fake $c$-foliation}. Instead, the fake $c$-foliation is defined by 
$\tW^c_x(y)=\tW^{cu}_x(y)\cap\tW^{cs}_x(y)$. By the previous lemma, it is invariant under $\tg_\cdot$ and is the graph of a function $\varphi: E^c _x\rightarrow E^{us}_x$ with slope $\leq \frac{\gamma_0}{C_1}$. It subfoliates both $\thF^{cs}_x$ and $\thF^{cu}_x$.

{
	\begin{remark}\label{r.foliation.size}
		It is worth pointing out that fake foliations are defined on the entirety of $T_xM$ and are invariant under $\tg_\cdot$; however,  for points that stay in the ``local charts'' $\{|v|\leq r_0\}$, we have $\tg_\cdot=\tf_\cdot$ and thus the local foliations are invariant under $\tf_\cdot$. In other words, fake foliations are invariant under the lift of $f$ as long as they remain within the uniform size $r_0$; only in this case do they capture the behavior of the actual dynamics of $f$. We call this the local invariance property.
	\end{remark}
}

For simplicity, when $y\in \exp_x\{|v|\leq r_0\}$, we write $\tW_x^i(y) :=  \tW^i_x(\exp_x^{-1}y)$ for $i\in\{s,c,u,cu,cs\}$.

\subsection{Uniformly Hölder holonomy}\label{section 2.2}
In this section, we {consider the regularity of the fake $u$-holonomies insider the fake $cu$-leaves}. When $f$ is only $C^1$, we lose the Lipschitz property of the holonomy maps as in \cite[Section 4.2]{LEDRAPPIER_YOUNG_A}; however, we can prove the uniformly Hölder property with Hölder exponent arbitrarily close to $1$. We first define $c$-transversal.

\begin{definition}[$c$-transversal]\label{c-trans-def}
	For each $x\in\Gamma'$, $U\subseteq T_xM$ is called a $c$-transversal inside $\tW^{cu}_x(x)\subseteq T_xM$ with slope $\leq \frac{1}{C}$, if it is a {$c$-dimensional submanifold} of $\hat{U}=\mathrm{graph}(\varphi)\subseteq \tW^{cu}_x(x)$ where the function $\varphi: E^c _x\rightarrow E^{us}_x$ satisfies $|D\varphi|\leq\frac{1}{C}$.  
\end{definition}

Then we can prove that the fake $u$-holonomies inside the fake $cu$-leaves are uniformly Hölder. Similar results can be found in \cite{AmieW2013} and \cite{CaoWZ}. {Recall $\epsilon$ from Lemma \ref{lemmaA(x)} which, with $\eta=\frac{\epsilon}{2}$, determines $\Gamma'$.}

\begin{lemma}[Uniformly Hölder]\label{uniform Hölder}
	For each $x\in\Gamma'$, given two $c$-transversals $U$ and $V$ with slope $\leq\frac{1}{C_1}$ inside $\tW^{cu}_x(x)\subseteq T_xM$, if the fake $u$-foliation $\thF ^{u}_x$ induces an injective holonomy $\tth^u_x:U\rightarrow V$, then the fake $u$-holonomy is uniformly Hölder. Specifically, for every $r>0$, if we have for any $ p\in U, |p-\tth^u_x(p)|\leq r$, then for $\alpha=1-\frac{7\epsilon}{\lambda^u}$, $C_2=C_2(A(x),r)$,
	\begin{equation*}
		C_2^{-1}|p-q|^{\alpha^{-1}}\leq |\tth^u_x(p)-\tth^u_x(q)|\leq C_2|p-q|^\alpha, \ \forall p,q\in U.
	\end{equation*}
\end{lemma}

\begin{proof}Fix $n\in\mathbb N$ whose choice (depending on $x,p$ and $q$) will be specified later.
	We define $U_n := \tg_x^{-n}U$, $V_n :=  \tg_x^{-n}V$ and note that $U_n,V_n\subseteq \tW^{cu}_{f^{-n}x}(f^{-n}x)$. By the definition of $C_1$ (see the discussion after Equation \eqref{local argument 2}), $U_n,V_n$ stay as $c$-transversals with slope $\leq \frac{1}{C_0}$. We denote $p_n=\tg^{-n}_xp$, $q_n=\tg^{-n}_xq$, and note that by invariance, $\tth^u_{f^{-n}x}p_n=\tg_x^{-n}\tth^u_xp$ and $\tth^u_{f^{-n}x}q_n=\tg^{-n}_x\tth^u_xq$.

	We denote $\hat{U}_n :=  \tg_x^{-n}\hat U$ and $\hat{V}_n :=  \tg_x^{-n}\hat V$ where $\hat U,\hat V$ are given in Definition \ref{c-trans-def}. In general, we write $d_R(\cdot,\cdot)$ for the distance in the submanifold $R$. 
	Since $\hat{U}_n$ and $\hat{V}_n$ are both graphs, we have $\frac{1}{2}|p-q|\leq d_{\hat U_n}(p,q)\leq 2|p-q|$ and $\frac{1}{2}|p-q|\leq d_{\hat V_n}(p,q)\leq 2|p-q|$.  Then we have
	\begin{equation}\label{local calclulate 1}
		\begin{aligned}
			|p_n-q_n|&=|\tg^{-n}_xp-\tg^{-n}_xq|\leq 2d_{\hat U_n}(\tg^{-n}_xp,\tg^{-n}_xq)\\
			&\leq 2\int^{p}_q |D\tg^{-n}_x(y)|_{T_y\hat{U}}|\,\td y\\
			&\leq \frac{1}{2}C_f^N\cdot\int^p_q |D\tg^{-KN}_x(y)|_{T_y\hat{U}}|\,\td y\ \ \ (\text{where }K=[\frac{n}{N}])\\
			&\leq \frac{1}{2}C^N_f\cdot\int^p_q |D\tf^{-KN}_x(0)|_{ E^c _x}|\cdot \mathrm{e}^{2KN\epsilon}\, \td y\ \ \ (\text{by  \Cref{local argument 1} and \ref{local argument 2}})\\
			&\leq \frac{1}{2}C^N_f \mathrm{e}^{2n\epsilon} A(x)\mathrm{e}^{KN\epsilon} d_{U}(p,q)\ \ \ (\text{by \Cref{ABC ergodic lemma}})\\
			&\leq A(x)C^N_f\mathrm{e}^{3n\epsilon} |p-q|.
		\end{aligned}     
	\end{equation}
	In the same way, we can estimate
	\begin{equation}\label{local calculate 2}
		\begin{aligned}
			|p_n-\tth^u_{f^{-n}x}(p_n)|&=|\tg^{-n}_xp-\tg^{-n}_x\tth^u_xp|\leq 2d_{\tW^u_{f^{-n}x}(p_n)}(\tg^{-n}_xp,\tg^{-n}_x\tth^u_xp)\\
			&\leq 2\int^{p}_{\tth^u_xp} |D\tg^{-n}_x(y)|_{T_y\tW^u_x(p)}|\td y\\
			&\leq \frac{1}{2}C_f^N\cdot\int^p_{\tth^u_xp} |D\tg^{-KN}_x(y)|_{{T_y\tW^u_x(p)}}|\td y\ \ \ (\text{where }K=[\frac{n}{N}])\\
			&\leq \frac{1}{2}C^N_f\cdot\int^p_{\tth^u_xp} |D\tf^{-KN}_x(0)|_{ E^u _x}|\cdot \mathrm{e}^{2KN\epsilon} \td y\ \ \ (\text{by  \Cref{local argument 1} and \ref{local argument 2}})\\
			&\leq \frac{1}{2}C^N_f \mathrm{e}^{2n\epsilon} A(x)\mathrm{e}^{KN(-\lambda^u+\epsilon)} d_{\tW^u_x(p)}(p,\tth^u_xp)\ \ \ (\text{by \Cref{ABC ergodic lemma}})\\
			&\leq A(x)C^N_f\mathrm{e}^{n(-\lambda^u+3\epsilon)} |p-\tth^u_xp| \\
			&\leq A(x)C^N_f\mathrm{e}^{n(-\lambda^u+3\epsilon)} \cdot r,
		\end{aligned}
	\end{equation}
	and similarly,
	\begin{equation*}
		|q_n-\tth^u_{f^{-n}x}(q_n)|\leq A(x)C^N_f\mathrm{e}^{n(-\lambda^u+3\epsilon)} \cdot r.         
	\end{equation*}

	We now choose a smooth foliation $\thF^{\bot}$ in $T_{ f^{-n}x}M$ with parallel leaves, such that each leaf is a plane parallel to the $E^{us}_{f^{-n}x}$ plane. This foliation induces a holonomy $h^{\bot}:\hat U_n\rightarrow\hat V_n$ which is a homeomorphism since $\hat U_n$ and $\hat V_n$ are graphs of functions with slopes less than $\frac{1}{C_0}$. Since the angle of the holonomy and the fake leaf is uniform, there exists $C_3>0$ such that
	\begin{equation*}
		|h^{\bot}(p_n)-h^\bot(q_n)|\leq C_3|p_n-q_n|\leq C_3 A(x)C^N_f\mathrm{e}^{3n\epsilon} |p-q|,
	\end{equation*}
	\begin{equation*}
		\begin{aligned}
			|\tth^u_{f^{-n}x}(p_n)-h^{\bot}(p_n)|&\leq C_3 |\tth^u_{f^{-n}x}(p_n)-h^{\bot}(p_n)|_{(E^{us}_x)^\bot}\\
			&=C_3|p_n-\tth^u_{f^{-n}x}(p_n)|_{(E^{us}_x)^\bot}\\
			&\leq C_3|p_n-\tth^u_{f^{-n}x}(p_n)| \\
			&\leq C_3 A(x)C^N_f\mathrm{e}^{n(-\lambda^u+3\epsilon)}\cdot  r,
		\end{aligned}
	\end{equation*}
	and in the same way,
	\begin{equation*}
		|\tth^u_{f^{-n}x}(q_n)-h^{\bot}(q_n)|\leq C_3 A(x)C^N_f\mathrm{e}^{n(-\lambda^u+3\epsilon)} \cdot r  .
	\end{equation*}
	Thus, we have that 
	\begin{equation*}
		|\tth^u_{f^{-n}x}(p_n)-\tth^u_{f^{-n}x}(q_n)|\leq C_3 A(x)C^N_f\mathrm{e}^{3n\epsilon} |p-q|+2C_3 A(x)C^N_f\mathrm{e}^{n(-\lambda^u+3\epsilon)} \cdot r.
	\end{equation*}
	Then we write
	\begin{equation}\label{local calculate 3}
		\begin{aligned}
			&  |\tth^u_x(p)-\tth^u_{x}(q)|\\&=|\tg^{n}_{f^{-n}x}\tth^u_{f^{-n}x}(p_n)-\tg^{n}_{f^{-n}x}\tth^u_{f^{-n}x}(q_n)|\\&\leq 2d_{\hat V_n}(\tg^{n}_{f^{-n}x}\tth^u_{f^{-n}x}(p_n),\tg^{n}_{f^{-n}x}\tth^u_{f^{-n}x}(q_n))\\
			&\leq 2\int^{\tth^u_{f^{-n}x}(p_n)}_{\tth^u_{f^{-n}x}(q_n)} |D\tg^{n}_{f^{-n}x}(y)|_{T_y\hat V_n(\tth^u_{f^{-n}x}(p_n))}|\,\td y\\
			&\leq \frac{1}{2}C_f^N\cdot\int^{\tth^u_{f^{-n}x}(p_n)}_{\tth^u_{f^{-n}x}(q_n)} |D\tg^{KN}_{f^{-n}x}(y)|_{T_y\hat V_n(\tth^u_{f^{-n}x}(p_n))}|\,\td y\ \ \ (\text{where }K=[\frac{n}{N}])\\
			&\leq \frac{1}{2}C^N_f\cdot\int^{\tth^u_{f^{-n}x}(p_n)}_{\tth^u_{f^{-n}x}(q_n)} |D\tf^{KN}_{f^{-n}x}(0)|_{ E^u _{f^{-n}x}}|\cdot \mathrm{e}^{2KN\epsilon} \,\td y\ \ \ (\text{by  \Cref{local argument 1} and \ref{local argument 2}})\\
			&\leq \frac{1}{2}C^N_f \mathrm{e}^{2n\epsilon} A(f^{-n}x)\mathrm{e}^{KN\epsilon} d_{\hat V_n}(\tth^u_{f^{-n}x}(p_n),\tth^u_{f^{-n}x}(q_n))\ \ \ (\text{by \Cref{ABC ergodic lemma}})\\
			&\leq A(f^{-n}x)C^N_f\mathrm{e}^{3n\epsilon} |\tth^u_{f^{-n}x}(p_n)-\tth^u_{f^{-n}x}(q_n)| \\
			&\leq A(x)C^N_f\mathrm{e}^{4n\epsilon}(C_3 A(x)C^N_f\mathrm{e}^{3n\epsilon} |p-q|+2C_3 A(x)C^N_f\mathrm{e}^{n(-\lambda^u+3\epsilon)} r).\\
		\end{aligned}
	\end{equation}
	Hence there exists a constant $C_4=C_4(A(x))$ such that 
	\begin{equation*}
		|\tth^u_x(p)-\tth^u_{x}(q)|\leq C_4 (\mathrm{e}^{7n\epsilon}|p-q|+\mathrm{e}^{(-\lambda^u+7\epsilon)n}r).
	\end{equation*}
	Letting $n=[\frac{\log |p-q|}{-\lambda^u}]$, $C_2=C_4C_f(r+1)$ and recalling the choice of $\alpha = 1-7\epsilon/\lambda^u$, we have
	\begin{equation*}
		|\tth^u_x(p)-\tth^u_x(q)|\leq C_2|p-q|^\alpha.
	\end{equation*}
	With this, we finished the proof for the right-hand side of the lemma. The left-hand side can be proven similarly.
\end{proof}
\begin{remark}[The local argument]\label{local argument}
	\Cref{local calclulate 1}, \Cref{local calculate 2} and \Cref{local calculate 3} are obtained by the same method, that is, by combining a form of continuity (\Cref{local argument 1}, \Cref{local argument 2}) with hyperbolicity (\Cref{ABC ergodic lemma}) to estimate the growth of the length of the fake leaves or $c$-transversals. This argument was used in \cite[Section 8]{ABC} and allows us to bypass the one-step estimate within Lyapunov charts in \cite{LEDRAPPIER_YOUNG_A} and nullify the lack of a higher regularity. We refer to this technique as \textbf{the local argument} and will use it frequently in the rest of this paper. For simplicity, henceforth we will omit the details and only state the result whenever the same argument is used.
\end{remark}

\subsection{Unstable manifolds and center unstable sets}\label{section 2.3}
In this section we discuss the geometric structure of the unstable manifolds and center unstable sets. We introduce the concept of fake foliation charts. For this purpose, let $\epsilon>0$ be fixed, and $\Gamma'$ be the full measure set given by Lemma \ref{ABC ergodic lemma}. Note that $R_x(\rho)$ is the ball centered at $0$ with radius $\rho$ under the box norm.

\begin{definition}[Fake foliation charts]\label{fake foliation charts def}
	For $x\in \Gamma'$, the $\delta$-fake foliation chart of $x$ is the pair $\{R_x(\delta A(x)^{-1}),\exp_x\}$ where $R_x(\delta A(x)^{-1})\subseteq T_xM$ and $\exp_x: R_x(\delta A(x)^{-1})\rightarrow M$.
\end{definition}

In Remark \ref{r.foliation.size} we established that fake foliations carry true dynamical information only within the uniform size $r_0$, inside which we have $\tilde g = \tilde f$. To guarantee that orbits stay within this scale, we will only consider foliation charts of the size $\delta A(x)^{-1}\ll r_0$, as indicated by the local argument (see for instance the proof of Lemma \ref{uniform Hölder} and \Cref{lemma 8 the size} below). This will also be the scale for which unstable/center unstable manifolds are defined. Remarkably, this is precisely the size of the Lyapunov charts in \cite{LEDRAPPIER_YOUNG_A} as the reciprocal of the hyperbolic function $A(x)$.



Recall that $C_f$ was chosen by Equation \eqref{e.cf} and $C_0$ was taken such that \eqref{local argument 2} holds. For $0<\delta\leq \delta_0 :=  \frac{r_0}{100C^N_fC_0^{10}}$, we define the center unstable set
$$S^{cu}_\delta(x) := \{v\in T_xM:|\tf^{-n}_xv|'\leq\delta A(f^{-n}x)^{-1},\forall n\geq 0\}.$$
The global unstable manifold of $x\in\Gamma'$
$$W^u(x)=\{ y\in M: \limsup_{n\rightarrow \infty}\frac{1}{n}\log (d(f^{-n}x,f^{-n}y))<0 \}$$
is an injectively immersed $C^1$ manifold with dimension $u$ by \cite[Proposition 8.9]{ABC}. 
For any $y\in\Gamma'\cap\exp_xS^{cu}_\delta(x)$ we denote 
$$
W^u_{x,2\delta}(y) := \exp^{-1}_x(\text{the component of }W^u(y)\cap \exp_xR_x(2\delta A(x)^{-1})\text{ that contains $y$})
$$ as the local unstable manifold (in the $\delta$-fake foliation charts of $x$). We first prove the following lemma, which characterizes the size of the fake unstable leaves that coincide with the true unstable manifolds inside the center unstable sets.  
\begin{lemma}\label{lemma 8 the size}
	For any $y\in\Gamma'\cap\exp_x S^{cu}_\delta(x)$, $W^u_{x,2\delta}(y)=\tW^u_x(y)\cap R_x(2\delta A(x)^{-1})$.
	\begin{proof}
		Since $W^u_{x,2\delta}(y)$ is a $u$-dimensional connected injectively immersed $C^1$ manifold and $\tW^u_x(y)\cap R_x(2\delta A(x)^{-1})$ is a graph of a function $\varphi:R_x^{u}(2\delta A(x)^{-1})\rightarrow R_x^{cs}(2\delta A(x^{-1}))$ with slope $\leq \frac{\gamma_0}{C_1}$, we only need to prove that $\tW^u_x(y)\cap R_x(2\delta A(x)^{-1})\subseteq W^u_{x,2\delta}(y)$, which only requires us to prove that $\tW^u_x(y)\cap R_x(2\delta A(x)^{-1})\subseteq \exp_x^{-1}W^u(y)$.

		If $z\in \tW^u_x(y)\cap R_x(2\delta A(x)^{-1})$, we consider $\tg^{-n}_xz\in \tW^u_{f^{-n}x}(f^{-n}y)$. By the local argument, we have 
		\begin{equation}\label{equation 26}
			\begin{aligned}
				|\tg^{-n}_xz-\tg^{-n}_x\exp_x^{-1}y|&\leq C^N_f\mathrm{e}^{(-\lambda^u+3\epsilon)n}A(x)|z-\exp_x^{-1}y|\\
				&\leq 4C_0\delta C^N_f\mathrm{e}^{(-\lambda^u+3\epsilon)n}\leq \frac{r_0}{10},
			\end{aligned}
		\end{equation}
		which goes to zero exponentially when $n\to +\infty$.

		Since $y\in \exp_xS^{cu}_\delta(x)$, for any $n\geq 0$, $\tf^{-n}_x\exp_x^{-1}y$ stays in the local charts $\{|v|\leq r_0\}$. Since $\tf_x=\tg_x$ when $|v|\leq r_0$, we have $|\tg_x^{-n}\exp_x^{-1}y|\leq \frac{r_0}{10}$ and hence $|\tg^{-n}_xz|<r_0$. Thus, for any $n\geq 0$, $\tg_x^{-n}z$ stays in the local charts $\{|v|\leq r_0\}$. Together with \Cref{equation 26} we have $|\tf^{-n}_xz-\tf^{-n}_x\exp_x^{-1}y|$  goes to zero exponentially when $n\to +\infty$. This means $z\in\exp_x^{-1}W^u(y)$, and consequently $\tW^u_x(y)\cap R_x(2\delta A(x)^{-1})\subseteq \exp_x^{-1}W^u(y)$.
	\end{proof}
\end{lemma}

For $0<\delta\leq \delta_0$, we prove the following lemma, which is partially parallel to \cite[Lemma 2.2.3]{LEDRAPPIER_YOUNG_A}.
\begin{lemma}\label{lemma 10 the manifold property}
	A. For any $ x\in\Gamma',\ y\in \Gamma'\cap \exp_xS^{cu}_\delta(x)$:\\
	1) $W^u_{x,2\delta}(y)$ is the graph of a function $\varphi:R^u_x(2\delta A(x)^{-1})\rightarrow R_x^{cs}(2\delta A(x)^{-1})$ with slope $\leq\frac{\gamma_0}{C_1}$.\\
	B. For $m$-a.e.$\ x\in\Gamma'$, if $y\in\Gamma'\cap \exp_xS^{cu}_\delta(x)$ and $fy\in \exp_{fx}S^{cu}_\delta(fx)$, then \\
	1) $\tf_xW^u_{x,2\delta}(y)\cap R_{fx}(2\delta A(fx)^{-1})\subseteq W^u_{fx,2\delta}(fy)$.\\
	2) $S^{cu}_{2\delta}(x)\cap\exp_x^{-1}W^u(y)\subseteq W^u_{x,2\delta}(y)$.
	\begin{proof}
		A.1) It directly follows from \Cref{lemma 8 the size} and \Cref{fake foliations}.\\
		B.1) It follows from \Cref{lemma 8 the size} and the invariance of the fake $u$-foliations.\\
		B.2) By Poincaré recurrence theorem, for $m$-a.e.$\ x\in\Gamma'$ we have $A(f^{-n}x)\nrightarrow 0$. If there exists $z\in S^{cu}_{2\delta}(x)\cap\exp^{-1}_xW^u(y)$ but $z\notin W^u_{x,2\delta}(y)$, we have $d(\tf^{-n}_xz,\tf^{-n}_x\exp_x^{-1}y)\rightarrow 0$. Hence for typical $x$, there exists $k=k(x)\in \mathbb{N}$ which is the smallest number such that $\tf_x^{-k}z\in W^u_{f^{-k}x,2\delta}(f^{-k}y)$. By the assumption of $z$, we have $k>0$, and hence $\tf^{-k+1}_xz\notin W^u_{f^{-k+1}x,2\delta}(f^{-k+1}y)$; by B.1), we have $\tf^{-k+1}_xz\notin R_{f^{-k+1}x}(2\delta A(f^{-k+1}x)^{-1})$, a contradiction.
	\end{proof}
\end{lemma}
Finally, we prove that the center unstable set is contained in the $cu$-fake leaf.
\begin{lemma}\label{lemma above}
	For any $x\in\Gamma'$, $S^{cu}_{2\delta}(x)\subseteq \tW^{cu}_x(x)$.
	\begin{proof}
		Suppose on the contrary there exists $y\in S^{cu}_{2\delta}(x)$ but $y\notin \tW^{cu}_x(x)$. Let $z\in\tW^{cu}_x(x)\cap \tW^s_x(y)$ and note $z\ne y$. By the definition of the center unstable set, we have $\tf^{-n}_xy\in S^{cu}_{2\delta}(f^{-n}x)$ for every $n\in \mathbb N$ and by invariance, $$\tf^{-n}_xz\in\tf^{-n}_x\tW^{cu}_x(x)\cap\tf^{-n}_x\tW^s_x(y).$$ Hence we have $|\tf^{-n}_xz|'\leq 2(1+\frac{\gamma_0}{C_1})\delta$ for every $n\in\mathbb N$. But then the local argument shows
			\begin{align*}
				|z-y| &= |\tf^n_{f^{-n}x}(\tf^{-n}_xz) - \tf^n_{f^{-n}x}(\tf^{-n}_xy)|\\
				&\le 2(1+\frac{\gamma_0}{C_1})\delta A(f^{-n}x)C_f^Ne^{n(\lambda^s+3\epsilon)}\\
				&\le 2(1+\frac{\gamma_0}{C_1})\delta A(x)C_f^Ne^{n(\lambda^s+4\epsilon)}\\
				& \to 0 \mbox{ as $n\to\infty$},
			\end{align*}
			which means $z = y$, a contradiction.

	\end{proof}
\end{lemma}
\begin{remark}
	The center unstable sets defined in \cite[Section 2.2]{LEDRAPPIER_YOUNG_A} are a priori just $u$-saturated sets. In our setting, \Cref{lemma above} provides more topological information by showing that they are contained in the center unstable fake leaves. This can be seen as a form of dynamical coherence.
\end{remark}

\subsection{Adapted partition}\label{section 2.4}
In this section, we establish the existence of finite entropy adapted partitions, similar to \cite[Section 2.4]{LEDRAPPIER_YOUNG_A}. See also \cite{SunTianC1pesinformula}.

\begin{definition}[Adapted partitions]
	A measurable partition $\tP$ is said to be adapted to $\{A(x),\delta\}$, if for $m$-a.e.$\ x\in\Gamma'$, we have $\tP^{+}(x)\subseteq \exp_xS^{cu}_\delta(x)$, where $\tP^{+} := \bigvee^{\infty}_{i=0}f^i\tP$.
\end{definition}
Here we need the adapted partitions to have the property that points in each atom all stay well within the $\delta$-fake foliation charts. The following lemma ensures the existence of a finite entropy partition adapted to $\{A(x),\delta\}$. 
\begin{lemma}\label{adapted partition exist}
	For any $\epsilon>0$, $A(x)$ and $0<\delta\leq \delta_0$ defined as before, there exists a finite entropy partition $\tP$ such that for $m$-a.e.$\ x\in\Gamma',$ it holds that  $\tP^{+}(x)\subseteq\exp_xS^{cu}_\delta(x)$.
	\begin{proof}
		We first fix $A_0>0$ such that $\Lambda := \Lambda_{A_0} := \{x\in\Gamma':A(x)\leq A_0\}$ has positive measure. Define a function $\psi:M\rightarrow \mathbb R^{u}$ as:
		\begin{equation*}
			\psi(x)=\left\{
			\begin{aligned}
				&\delta  , &x&\notin \Lambda,\\
				&\delta A_0^{-1}C^{-r(x)}_f\frac{\mathrm{e}^{-\epsilon r(x)}}{2C_0}  , & x&\in \Lambda,
			\end{aligned}
			\right.
		\end{equation*}
		where $r(x)$ is the smallest positive integer $k$ such that $f^kx\in \Lambda$. 
		Following the same argument as in \cite[Lemma 2]{Mañé_1981} or  \cite[Lemma 2.4.2]{LEDRAPPIER_YOUNG_A}, there exists a finite entropy partition $\tP$ such that $\tP(x)\subseteq B(x,\psi(x))$ for almost every $x$, where $B(x,r)$ denote the ball centered at $x$ with radius $r$.

		It only remains to show $\tP^+(x)\subseteq \exp_xR_x(\delta A(x)^{-1})$ for $m$-a.e.$\ x$. 
		When $x\in\Lambda, \tP^+(x)\subseteq \tP(x)\subseteq B(x,\frac{\delta A_0^{-1}}{2C_0})\subseteq \exp_xR_x(\delta A(x)^{-1})$.
		When $x\notin \Lambda$, we choose $n>0$ as the smallest positive integer such that $f^{-n}x\in\Lambda$. Then $f^{-n}\tP^+(x)\subseteq B(f^{-n}x,\psi(f^{-n}x))$. Notice that by definition, $n\leq r(f^{-n}x)$. We have
		\begin{equation*}
			\begin{aligned}
				\tP^+(x)&\subseteq f^nB(f^{-n}x,\psi(f^{-n}x))\\
				&\subseteq B(x,C^n_f\psi(f^{-n}x))\\
				&\subseteq \exp_xR_x(2C_0C^n_f\delta A_0^{-1}C^{-r(f^{-n}x)}_f\frac{\mathrm{e}^{-\epsilon r(f^{-n}x)}}{2C_0})\\
				&\subseteq \exp_xR_x(\delta A_0^{-1} \mathrm{e}^{-\epsilon n})\\
				&\subseteq \exp_xR_x(\delta A(x)^{-1}).
			\end{aligned}
		\end{equation*}
		Here the last inclusion is due to
			$$
			A_0e^{\epsilon n} \ge A(f^{-n}x) e^{\epsilon n} \ge A(x).
			$$
			The proof is now complete.
	\end{proof}
\end{lemma}

\section{Special partitions and unstable entropy}\label{u-entropy}
This section contains the basic properties of $u$-entropy.
In \Cref{section 3.1} we introduce the $u$-entropy for $C^1$ diffeomorphisms with dominated splitting; in \Cref{section 3.2} we define two special partitions and discuss their quotient structure; then in \Cref{section 3.3} we introduce a way to define transverse metrics. Most of the results here are parallel to \cite[Section 3]{LEDRAPPIER_YOUNG_A}.

\subsection{The unstable entropy}\label{section 3.1}
In this section we construct measurable partitions subordinate to $W^u$ following \cite{Ledrappier_Strelcyn_1982}, and use them to define the $u$-entropy of $f$.

Let $A(x)$ and $\Lambda_{A_0}$ be as before, and, parallel to \cite{Ledrappier_Strelcyn_1982}, write $V_{loc}(y)=\exp_yW^u_{y,\delta_0}(y)$ for $y\in\Lambda_{A_0}$. We now check that all the conditions in \cite[Section 3.3]{Ledrappier_Strelcyn_1982} hold. For simplicity, we provide only brief explanations and leave details to the reader.

\textit{1. For each $y\in\Lambda_{A_0}$, there exist $\epsilon({A_0})>0$ and $r_{A_0}>0$ such that for any $0<r\leq r_{A_0}$ and any $y\in\Lambda_{A_0}\cap B(x,\epsilon({A_0})\cdot r)$, $V_{loc}(y)\cap B(x,r)$ is connected.}
This follows from \Cref{lemma 8 the size} and the local graph structure of the fake foliations, \Cref{fake foliations}.

\textit{2. $y\mapsto V_{loc}(y)\cap B(x,r_{A_0})$ is continuous at every $y\in\Lambda_{A_0}\cap B(x,\epsilon(A_0)r_{A_0})$.}
This follows from \Cref{lemma 8 the size} and the continuity of the fake foliation, \Cref{fake foliations}.

\textit{3. There exists $ A_1>0$ such that $V_{loc}(y)$ contains the ball of radius $A_1$ centered at $y$.}
This follows from \Cref{lemma 8 the size} and \Cref{fake foliations}, since $A_0$ is uniform.

\textit{4. There exist $B_{A_0}>0,C_{A_0}>0$ such that if $z\in V_{loc}(y)$ then for every $n\geq 0$, $d(f^{-n}y,f^{-n}z)\leq B_{A_0}\mathrm{e}^{-nC_{A_0}}d(y,z)$.}
This follows from the proof of \Cref{lemma 8 the size}, more specifically, \Cref{equation 26}.

\textit{5. For any point $a_0$ in the support of $m|_{\Lambda_{A_0}}$ and any $0<r\leq r_{A_0}$, if two points $z_1,z_2\in S(a_0,r) := \bigcup_{y\in\Lambda_{A_0}\cap B(a_0,\epsilon({A_0})r)} V_{loc}(y)\cap B(a_0,r)$ are not in the same local unstable leaf, then $d^u(z_1,z_2)>2r$, where $d^u$ means the distance along the global unstable manifolds. If $z_1,z_2$ do not lie in the same global unstable manifold, $d^u(z_1,z_2) = \infty$.} This follows from \Cref{lemma 8 the size} and the local structure of the fake foliation, \Cref{fake foliations}.

Furthermore, for any point $a_0\in\supp(m|_{\Lambda_{A_0}})$, by reducing $r_{A_0}$, we can ensure that for every $S(a_0,r)$ defined as above and for any $x\in S(a_0,r)\cap \Lambda_{A_0},\ y\in S^{cu}_{\delta_0}(x)$, we have 
\begin{equation*}
    V_{loc}(y)\cap B(a_0,r)\subseteq \exp_xW^u_{x,2\delta_0}(y).
\end{equation*}

Hence following \cite{Ledrappier_Strelcyn_1982} (see also \cite[Section 3.1]{LEDRAPPIER_YOUNG_A}), we can construct a measurable set $S=S(a_0,r)$ as above (but not unique) such that the following holds:\\
a) $m(S\cap \Lambda_{A_0})>0$.\\
b) $S$ is a disjoint union of continuous embedded disks $\{D_\alpha\}$; each $D_\alpha$ is an open subset of $\exp_{x_\alpha} W^u_{x_\alpha,\delta_0}(x_\alpha)$ where $x_\alpha\in \Lambda_{A_0}$.\\
c) for $m$-a.e.$\ x\in\Gamma'$, there exists an open neighborhood $U_x$ of $x$ in $W^u(x)$ such that for any $n>0$, either $f^{-n}U_x\cap S=\emptyset$ or $f^{-n}U_x\subseteq D_\alpha$ for some $\alpha$.\\
d) Denote $r_1 :=  \sup_{\alpha} d^u(D_\alpha)$. If $x,y\in S$ lie on different local unstable leaves, then $ d^u(x,y)>r_1$.

Then we introduce two measurable partitions $\hat\xi$ and $\xi$: 
\begin{equation}\label{e.30}
\hat{\xi}(x)=\left\{
    \begin{aligned}
        &D_\alpha, & x\in D_\alpha\\
        &M-S,      & x\notin S,
    \end{aligned}
    \right .
\end{equation}
and
\begin{equation*}
    \xi(x)=\hat\xi^+(x).
\end{equation*}

The following lemma follows directly from the definition of $\xi$, \cite{Ledrappier_Strelcyn_1982} and the properties a)--d) mentioned above. The proof is left to the reader.
\begin{lemma}
    For Leb-a.e. $r>0$ small enough, the measurable partition
    $\xi$ defined as above is subordinate to $W^u$.
\end{lemma}

We then define the $u$-entropy through the entropy of the partitions subordinate to $W^u$ as $h^u_m(f) :=  h_m(f,\xi)=H(\xi|f\xi)$. The following lemma shows that the $u$-entropy does not depend on the choice of $\xi$, and is a number only determined by $f$ and $m$.
\begin{lemma}[{\cite[Lemma 3.1.2]{LEDRAPPIER_YOUNG_A}}]
    For any partitions $\xi_1,\xi_2$ subordinate to $W^u$, it holds that $$h_m(f,\xi_1)=h_m(f,\xi_2) :=  h^u_m(f).$$
\end{lemma}

\subsection{Special partitions and the quotient structure}\label{section 3.2}
In this section we construct two special partitions following \cite[Section 3.2]{LEDRAPPIER_YOUNG_A}; one will be used to calculate the $u$-entropy and the other is to calculate the metric entropy. We also discuss the quotient structure between them.

Now assume $\delta\leq\min\{\delta_0,\frac{r_1}{100C^N_fC_0^{10}}\} := \delta_1$. We choose a finite entropy partition $\tP$ adapted to $(A(x),\delta)$ as in \Cref{adapted partition exist} and we further require that $\tP$ refines $\{S,M-S\}$.
Define $\eta_1=\xi\vee\tP^+$ and $\eta_2=\tP^+$. With the same notations as before, we have:
\begin{lemma}\label{basic property of eta1 eta2}
The following properties hold:\\
    1) $\eta_1$ and $\eta_2$ are increasing, which means that $f\eta_1\prec\eta_1$ and $f\eta_2\prec\eta_2$.\\
    2) $\eta_2\prec\eta_1$.\\
    3) $\eta_2(x)\subseteq\exp_xS^{cu}_\delta(x)$, $\eta_1(x)\subseteq\exp_xW^u_{x,2\delta}(x)$ for $m$-a.e.$\ x$.\\
    4) $h_m(f,\eta_2)=h_m(f,\tP)$ and $h_m(f,\eta_1)=h^u_m(f)$.
\begin{proof}
    1), 2) and the first part of 3) are obvious.
    The second part of 3) follows from \Cref{lemma 10 the manifold property} B.2).
    See \cite[Section 3.2]{LEDRAPPIER_YOUNG_A} for details of the proof of 4).
\end{proof}
\end{lemma}

Below we present the quotient structure of $\eta_2(x) /\eta_1 $. The following lemmas are similar to \cite[Lemma 3.3.1]{LEDRAPPIER_YOUNG_A} and \cite[Lemma 3.3.2]{LEDRAPPIER_YOUNG_A}.
\begin{lemma}\label{lemma 28}
    For $m$-a.e.$\ x$ and any $y\in\Gamma'\cap\eta_2(x)$, $\exp_xW^u_{x,2\delta}(y)\cap \eta_2(x)=\eta_1(y)$.
\begin{proof}\label{first quotient structure}
    We first prove that $\exp_xW^u_{x,2\delta}(y)\cap \eta_2(x)\subseteq\eta_1(y)$.
    If $z\in \exp_xW^u_{x,2\delta}(y)\cap \eta_2(x)$, we only need to show that $z\in\xi(y)$, which holds if $d^u(f^{-n}y,f^{-n}z)\leq r_1$ whenever $f^{-n}y\in S$ (see Item d) above Equation \eqref{e.30}). Since $y,z\in\eta_2(x)$, by the local argument, we have
    \begin{equation*}
    \begin{aligned}
        d^u(f^{-n}y,f^{-n}z)
        &\leq 4|\exp_{f^{-n}x}^{-1}f^{-n}y-\exp_{f^{-n}x}^{-1}f^{-n}z|\\
        &\leq 4C^N_f\mathrm{e}^{(-\lambda^u+3\epsilon)}A(x)|\exp_x^{-1}y-\exp_x^{-1}z|\\
        &\leq 4C_0C^N_f\mathrm{e}^{(-\lambda^u+3\epsilon)}A(x)\cdot 2\delta A(x)^{-1}\\
        &\leq 8C_0C^N_f\mathrm{e}^{(-\lambda^u+3\epsilon)}\frac{r_1}{100C^N_fC_0^{10}}\\
        &\leq r_1.
    \end{aligned}
    \end{equation*}
    Hence it follows that $\exp_xW^u_{x,2\delta}(y)\cap \eta_2(x)\subseteq\eta_1(y)$.

    Then we prove that $\eta_1(y)\subseteq \exp_xW^u_{x,2\delta}(y)\cap \eta_2(x)$. $\eta_1(y)\subseteq \eta_2(x)$ is obvious, and $\eta_1(y)=\xi(y)\cap \eta_2(y)\subseteq \exp_xW^u_{x,2\delta}(y)$ follows from the construction of $\xi$, Lemma \ref{basic property of eta1 eta2} 3) and \Cref{lemma 10 the manifold property} B.2).
\end{proof}
\end{lemma}

\begin{lemma}\label{quotient structure}
For $m$-a.e.$\ x$ and any $y\in\Gamma'\cap \eta_2(x)$, $f^{-1}(\eta_1(y))=\eta_1(f^{-1}y)\cap f^{-1}(\eta_2(x))$.
\begin{proof}
    By \Cref{basic property of eta1 eta2}, we have $f^{-1}\eta_1(y)\subseteq f^{-1}\eta_2(y)$ and $f^{-1}\eta_1(y)\subseteq \eta_1(f^{-1}y)$. Hence we have  $f^{-1}(\eta_1(y))\subseteq\eta_1(f^{-1}y)\cap f^{-1}(\eta_2(x))$.

    To prove $\eta_1(f^{-1}y)\cap f^{-1}(\eta_2(x))\subseteq f^{-1}(\eta_1(y))$, it suffices to show
    \begin{equation*}
      f(\eta_1(f^{-1}y))\cap \eta_2(x)\subseteq \eta_1(y)=\exp_xW^u_{x,2\delta}(y)\cap \eta_2(x)  
    \end{equation*}
    then apply the previous lemma. Since $\eta_1(f^{-1}y)\subseteq \exp_{f^{-1}x}W^u_{f^{-1}x,2\delta}(y)$, we only need 
    \begin{equation*}
        f(\exp_{f^{-1}x}W^u_{f^{-1}x,2\delta}(y))\cap \eta_2(x)\subseteq \exp_xW^u_{x,2\delta}(y)\cap \eta_2(x);
    \end{equation*}
    this follows directly from \Cref{lemma 10 the manifold property} B.1).
\end{proof}
\end{lemma}

\subsection{Transverse metrics}\label{section 3.3}
Next, we describe a way to define transverse metrics through recurrence. As before, we use $u,c,s$ to denote the dimension of their corresponding subbundles.

Choose $E\subseteq S= \bigsqcup_\alpha D_\alpha$ such that $m(E)>0$ and let $\tau:D^{c+s}\rightarrow M$ be a $C^1$ embedding such that  $T :=  \mathrm{Im}\tau$ is transverse to each $D_\alpha\in S$, and $T\cap D_\alpha$ has exactly one point whenever $D_\alpha\cap E\neq\emptyset$. Here $D^{c+s}$ denotes a $(c+s)$-dimensional open unit disk. We further require that $\tP$, the finite entropy partition constructed at the beginning of this section, refines $\{E,M-E\}$.

Now we define $\pi^u:\bigcup_{n\geq 0}f^nE\rightarrow T$ as follows. 
When $x\in E\cap D_\alpha$, $\{\pi^u(x)\}=T\cap D_\alpha$. If $x\in (\bigcup_{n\geq 0}f^nE)- E$, let $\pi^u(x)=\pi^u(f^{-n_0(x)}x)$, where $n_0(x)>0$ is the smallest number such that $f^{-n_0(x)}x\in E$. Since $m$ is ergodic, $m(\bigcup_{n\geq 0}f^nE)=1$.

For $x\in\bigcup_{n\geq 0}f^nE$ and $y_1,y_2\in\eta_2(x)$, we define the transverse metric $d^T_x(\cdot,\cdot)$ inside $\eta_2(x)$ as $d^T_x(y_1,y_2)=|\tau^{-1}\pi^u y_1-\tau^{-1}\pi^u y_2|$. Since $\tP$ refines $\{E,M-E\}$ and $\eta_2 = \tP^+$, we see that $n_0(x)$, as defined above, is constant on each atom of $\eta_2$. In particular, the transverse metric on each atom of $\eta_2$ does not depend on the choice of $x$.

\section{Uncentered maximal functions}\label{section 4}


Recall that for general Radon measures, the Lebesgue density theorem is proven using centered maximal functions as in \cite[Section 4]{LEDRAPPIER_YOUNG_A}. In this paper, however, we consider the {\em uncentered maximal operator} as in \cite[Chapter 3]{Stein_book} to bypass the lack of Lipschitz regularity for the holonomy maps. In dimension one, it is well-known (see, for instance, \cite{PS83}) that the uncentered maximal operator is of weak type $(1,1)$, and therefore Lebesgue density theorem holds for uncentered balls. This is precisely where the one-dimensional hypothesis (or the multiplicity-one property of the Lyapunov exponents) is used.

Below, we only consider open coverings on $\mathbb R$, that is, $n=1$. In this case,  an uncentered ball w.r.t.\ $y$ is just a connected open interval that contains $y$. The next lemma has been used to prove \cite[Proposition 3.2]{Ledrappier_Strelcyn_1982}; the proof is omitted.

\begin{lemma}[One-dimensional covering lemma]\label{OCL} Assume that $E\subseteq \mathbb R$ or $E\subseteq S^1$ and $\tA$ is a finite  covering of $E$ by open intervals. Then there exists a subcovering $ \tA'\subseteq\tA$ of $E$ such that every $x\in E$ is contained in at most two intervals.
\end{lemma}

Now let $m$ be a probability measure on $\mathbb R$, $g\in L^1(m),g\geq 0$. We define the {\em uncentered} average, maximal, and minimal function as:
\begin{equation*}
    g^B(x)=\frac{1}{m(B)}\int_{B} g(z)\,\td m(z)\, \mbox{ where $B$ is an uncentered ball w.r.t.\ $x$,}
\end{equation*}
and
\begin{equation*}
    g^*(x)=\sup_{B\ni x}g^B(x)
\end{equation*}
\begin{equation*}
      g_*(x)=\inf_{B\ni x}g^B(x)
\end{equation*}
It is easy to see that $g^*$ (resp.\ $g_*$) is lower (resp.\ upper) semi-continuous and hence measurable.  

\begin{lemma}[Maximal-minimal functions estimates] For any $\lambda>0$, the following hold: \\
a) $\displaystyle m(g^*>\lambda)\leq \frac{2}{\lambda}\int g\,\td m$.\\
b) $\displaystyle\int_{\{g_*<\lambda\}}g\,\td m \leq 2\lambda$.
\begin{proof}
    a) is essentially \cite[Theorem (a)]{PS83} (with the constant $2$ instead of $5$ on the right-hand side due to Lemma \ref{OCL}) and is omitted. Below we prove b).


    Let $\mathfrak B=\{g_*<\lambda\}$. For every $x\in \mathfrak B$, we take $ B_x\ni x$ such that $g^{B_x}(x)<\lambda$, that is, $\int_{B_x}g\,\td m<\lambda m(B_x)$. Then $\{B_x:x\in\mathfrak B\}$ is an open covering of $\mathfrak B$, and there exists a countable subcovering $\{B_i:i\in \mathbb N\}$.

    Now $\displaystyle\int_\mathfrak{B}g\,\td m\leq\int_{\cup^\infty_{i=1}B_i}g\,\td m$ and we have $\displaystyle\int_{\cup^N_{i=1}B_i}g\,\td m \to \int_{\cup^\infty_{i=1}B_i}g\,\td m$ as $N\to\infty$. For each $N\in\mathbb N$, we treat $\{B_1,B_2\dots B_N\}$ as a finite subcovering of $\bigcup_{i=1}^N B_i$, and apply Lemma \ref{OCL} to obtain a finite subcovering $\{B_{i_1}, \cdots, B_{i_{M(N)}}\}$ of $\bigcup_{i=1}^N B_i$ that covers each point at most twice. We have
    \begin{align*}
		\int_{\cup^N_{i=1}B_i}g\,\td m \leq \sum_{j=1}^{M(N)}\int_{B_{i_j}}g\,\td m\leq \sum_{j=1}^{M(N)}\lambda m(B_{i_j})\leq 2\lambda
    \end{align*} 
	for every $N\in\mathbb N$. Sending $N\to\infty$, we have $\int_{\{g_*<\lambda\}}g\,\td m \leq 2\lambda$ as required.
\end{proof}
\end{lemma}

With this lemma, we can prove the following lemma as in \cite[Section 4.1]{LEDRAPPIER_YOUNG_A}.
\begin{lemma}\label{1 dim convergence lemma}
    Let $(X,m)$ be a Lebesgue space and $\pi:X\rightarrow\mathbb R$ a measurable map. Denote by $\{m_t\}_{t\in \mathbb R}$ the conditional measure of $m$ along the fibers $\pi^{-1}(t), t\in\mathbb R$. Let $\beta$ be a finite entropy partition of $X$. Given $t\in\mathbb R, A\in\beta$, we define $g^A(t)=m_t(A)$. $g^A:\mathbb R\rightarrow\mathbb R$.

    We define $g,g^B$ and $g_*:X\rightarrow\mathbb R$ as follows:
    \begin{align*}           &g(x)=\sum_{A\in\beta} \chi_A(x)g^A(\pi x),\\
    &g^B(x)=\sum_{A\in\beta} \chi_A(x)(g^A)^B(\pi x),\\ 
    &g_*(x)=\sum_{A\in\beta} \chi_A(x)(g^A)_*(\pi x),
    \end{align*}
    where $B$ is an uncentered ball w.r.t.\ $\pi x$.
    Then:\\
    a) $g^B(x)\to g(x)\ m$-a.e.\ as $\diam B\to 0.$\\
    b) $\int -\log  g_*\,\td m\leq H_m(\beta)+\log 2+1$.
    \begin{proof}
        See \cite[Lemma 4.1.3]{LEDRAPPIER_YOUNG_A}.
    \end{proof}
\end{lemma}

Finally, we recall a classical result in dimension theory. 
\begin{lemma}\label{dimension lemma}
    Let $m$ be a finite Borel measure on $\mathbb R^n$. Then for $m$-a.e.\ $x$, one has  \begin{equation*}
        \limsup_{\epsilon\rightarrow0} \frac{\log m(B(x,\epsilon))}{\log \epsilon}\leq n.
    \end{equation*}
\end{lemma}

\section{Proof of Main Theorem}\label{section 5}
In this section, we finish the proof of the Main Theorem. In \Cref{section 5.1} we give requirements for the size of a set $E$ to ensure some geometry properties. In \Cref{section 5.3} we prove the convergence arguments (\Cref{convergence arguments}) for a special family of uncentered balls. In \Cref{section 5.4} we finish the proof of the Main Theorem.

Before starting the proof, we briefly compare our proof in \Cref{section 5} with \cite[Section 5]{LEDRAPPIER_YOUNG_A} and explain how we circumvent the lack of Lipschitz regularity for the holonomy maps.

In \cite{LEDRAPPIER_YOUNG_A}, the authors use the Besicovitch covering lemma for centered balls; consequently, when they use the transverse metrics to define balls and iterate along the orbit of $x$, they need to use holonomy maps whenever $f^ix$ returns to $E$ (defined in \Cref{section 3.3} as in \cite[Section 3]{LEDRAPPIER_YOUNG_A}) to cover the image by a larger ball. This is to ensure that the images are always centered balls so the covering lemma is applicable. By the Birkhoff ergodic theorem, along a typical orbit of length $n$, the holonomy need to be used for approximately $m(E)n$ many times.

In our proof, the holonomy is only used twice (at $k=0$ and $k=n$). For intermediate returns to $E$, we use one-dimensional, uncentered balls when considering the maximal function, and the Besicovitch covering lemma is replaced by a one-dimensional covering lemma by uncentered intervals, namely Lemma \ref{OCL}. This is why the assumption $\dim E^c = 1$ is needed.

\subsection{Preparations}\label{section 5.1}
Let $E\subseteq S= \bigsqcup_\alpha D_\alpha$ be the positive measure set taken in Section \ref{section 3.3}, for which the measurable partitions $\eta_1,\eta_2$ and the projection $\pi^u$ are defined. In this section we specify certain requirements on how small $E$ must be so that some local geometry arguments hold.

With the same notations as in the previous sections, we set  $S'=S\cap\Lambda_{A_0}$. Choose a density point $\omega_0\in B(a_0,\epsilon(l)r)\cap \Lambda_{A_0}$. Following the definition of $\Lambda_{A_0}$, we can choose $E$ containing $\omega_0$ and of sufficiently small diameter, such that for some $\delta_{A_0}>0$ depending only on $A_0$, we have 
$$
E\subseteq \exp_xR_x(\delta_{A_0}A_0^{-1})
$$
for every $x\in E$. Such an $E$ exists once $\delta_{A_0}$ is fixed; the precise choice of $\delta_{A_0}$ will be provided later in this section; see \Cref{remark 6}.

We let $\tau :=  \exp_{\omega_0}|_{\{0\}\times R_{\omega_0}^{cs}(\delta_0A^{-1}_0)}$ and $T :=  \mathrm{Im}\tau$ as in \Cref{section 3.3}. For this $T$, there exists a constant $C_5>0$ such that 
\begin{equation*}
    C_5^{-1}d^T_{R^{c+s}}(\cdot,\cdot)\leq d^T_M(\cdot,\cdot)\leq C_5d^T_{R^{c+s}}(\cdot,\cdot)
\end{equation*}
where $ d^T_M(\cdot,\cdot)$ is the submanifold metric on $T$ and $d^T_{R^{c+s}}(\cdot,\cdot)$ is the $(c+s)$-dimensional Euclidean metric on $R^{cs}_{\omega_0}(\delta_0A_0^{-1})$ (instead of the box norm). Recall that $\delta_1\ll r_1\ll\delta_0$ where $\delta_1$ was defined in Section \ref{section 3.2}.

We now require $\delta_{A_0}<\frac{\delta_1}{10^6}$ sufficiently small so that for any $x\in E$, $\exp^{-1}_xT\cap R_x(\frac{1}{2}\delta_0A_0^{-1})$ is given by the graph of a function $\varphi:R_x^{cs}(\frac{1}{2}\delta_0A_0^{-1})\rightarrow R_x^u(\frac{1}{2}\delta_0A_0^{-1})$ with slope $\leq \frac{\gamma_0}{C_1}$. For sufficiently small $\delta_{A_0}$, 
$T$ is transverse to each $D_\alpha\subset S$ and $T\cap D_\alpha$ consists of exactly one point whenever the intersection is not empty.


Note that $\tW^{cu}_x(x)\cap R_x(\tfrac{1}{2}\delta_0A_0)$ and consequently 
	$\exp_x^{-1}T\cap\tW^{cu}_x(x)\cap R_x(\tfrac{1}{2}\delta_0A_0^{-1})$
	are both graphs of functions with slope $\leq \tfrac{\gamma_0}{C_1}$, the first from 
	$R_x^{cu}$ to $R_x^s$ and the second from $R_x^c$ to $R_x^{su}$ (all of scale $\frac{1}{2}\delta_0A_0^{-1}$). 
	We denote the $\exp_x$-image of the latter by $\overline L^c_x\subset T$, a one-dimensional $C^1$ manifold given by the intersection of $T$ with the $\exp_x$-image of $\tW^{cu}_x(x)$. 
	For sufficiently small $\delta_{A_0}$, 
	$\tau^{-1}\overline L^c_x\cap R_{\omega_0}(\tfrac{1}{4}\delta_0A_0^{-1})$ 
	is also the graph of a function $\varphi:R_{\omega_0}^c(\frac{1}{4}\delta_0A_0^{-1})\rightarrow R_{\omega_0}^{s}(\frac{1}{4}\delta_0A_0^{-1})$ with slope $\leq \frac{2\gamma_0}{C_1}$; we set 
	$L^c_x:=\overline L^c_x\cap \tau(R_{\omega_0}(\tfrac{1}{4}\delta_0A_0^{-1}))$.

We then consider the $u$-holonomy inside $\exp_xS^{cu}_\delta(x)\cap E\subseteq \exp_x\tW^{cu}_x(x)$ with $0<\delta\leq\frac{\delta_1}{100}$. For  $y\in \exp_xS^{cu}_\delta(x)\cap E$, $W^u_{x,2\delta_0}(y)$ is the graph of a function $\varphi:R^u_x(2\delta_0A_0^{-1})\rightarrow R_x^{cs}(2\delta_0A_0^{-1})$ with slope $\leq\frac{\gamma_0}{C_1}$. Since $\exp_x^{-1}\omega_0\in R_x(\delta_{A_0}A_0^{-1})\cap\exp_x^{-1}T$ and $\exp_x^{-1}y\in R_x(\frac{\delta_1}{100}A_0^{-1})$, the intersection 
$$
W^u_{x,2\delta_0}(y)\cap \exp_x^{-1}T\cap R_x\left(\frac{1}{4}\delta_0A_0^{-1}\right)    
$$
consists of a unique point, namely $\exp_x^{-1}\pi^u(y)$ as constructed in Section \ref{section 3.3}. 
For small enough $\delta_{A_0}$ we have 
$\tau^{-1}\pi^u(y)\in R_{\omega_0}(\tfrac{1}{4}\delta_0A_0^{-1})$, hence 
$\pi^u(\exp_xS^{cu}_\delta(x)\cap E)\subseteq L^c_x$.

Similarly, the same argument shows that if we denote the local fake leaves by $\tW^{*}_{x,\delta}(x) :=  \tW^{*}_x(x)\cap R_x(\delta A_0^{-1})$, $*\in\{s,c,u,cu,cs\}$, and the fake $u$-holonomy by $\tpi^u_x:\tW^{cu}_{x,\delta}(x)\rightarrow T$, then we have
\begin{equation*}
    \tpi^u_x(\tW^{cu}_{x,\delta}(x))\subseteq L^c_x
\end{equation*}
\Cref{lemma 8 the size} shows that $\tpi^u_x=\pi^u\circ\exp_x := \pi^u_x$ when restricted to $S^{cu}_\delta(x)\cap\exp_x^{-1}E$. By \Cref{uniform Hölder}, we have that $\exp^{-1}_x\tpi^u_x:\tW^{cu}_{x,\delta}(x)\cap\tW^{cs}_{x,\delta}(x)=\tW^c_{x,\delta}(x)\rightarrow\exp_x^{-1}L^c_x $ is uniformly $\alpha$-Hölder with $\alpha=1-\frac{7\epsilon}{\lambda^u}$ and Hölder constant $C_2=C_2(A_0)$.

\begin{remark}\label{remark 6}
To summarize, we need $\delta_{A_0}$ small enough to ensure that the following properties hold at every $x\in E$ (fixing a density point $\omega_0\in B(a_0,\epsilon(l)r)\cap \Lambda_{A_0}$):
\begin{enumerate}
\item[1.] $T = \mathrm{Im} \tau =  \exp_{\omega_0}\left({\{0\}\times R_{\omega_0}^{cs}(\delta_0A^{-1}_0)}\right)$ is transverse to each $D_\alpha\subseteq S$, and $T\cap D_\alpha$ consists of at most one point.
\item[2.]  $\tau^{-1}L^c_x$ is the graph of a smooth function $\varphi:R^c_{\omega_0}(\frac{1}{4}\delta_0A_0^{-1})\rightarrow R^{s}_{\omega_0}(\frac{1}{4}\delta_0A_0^{-1})$ with slope $\leq \frac{2\gamma_0}{C_1}$.

\item[3.] $\pi^u(\exp_xS^{cu}_\delta(x)\cap E)\subseteq L^c_x$, and $\tpi^u_x(\tW^{cu}_{x,\delta}(x))\subseteq L^c_x.
$
\end{enumerate}
\end{remark}

Finally, we choose $E$ of sufficiently small diameter and $0<\delta\leq\frac{\delta_1}{100}$ satisfying all the above conditions. By \Cref{adapted partition exist}, there exists a finite entropy measurable partition $\tP$ that refines both $\{S,M-S\}$ and $\{E,M-E\}$, is adapted to $\{A(x),\delta\}$, and satisfies $h_m(f,\tP)\geq h_m(f)-\epsilon$. We then define $\eta_1 :=  \xi\vee\tP^+$ and $\eta_2 :=  \tP^+$ as in \Cref{section 3.2}. By \Cref{section 3.2}, $\eta_2(x)/\eta_1$ has a quotient structure, and by \Cref{section 3.3}, the transverse metric $d^T(\cdot,\cdot)$ is defined for $m$-a.e.\ $x\in M$. 

Next we consider the uncentered balls for $y\in \bigcup_{n\ge 0}f^nE$. For this purpose, let $y\in \bigcup_{n>0}f^nE$ and $B\subset \tau^{-1}T$ be any ball that contains $\tau^{-1}\pi^uy$. We define
\begin{equation*}
    B^T:= \left\{z\in\bigcup_{n\geq 0}f^nE:n_0(z) = n_0(y),\mbox{ and }\tau^{-1}\pi^uz\in B \right\},
\end{equation*}
where, as in Section \ref{section 3.3}, $n_0(z)>0$ is the smallest number such that $f^{-n_0(z)}x\in E$.

\subsection{The convergence argument}\label{section 5.3}
Given an uncentered ball $B$ containing $\tau^{-1}\pi^uy$, we define two functions: $g:M\rightarrow\mathbb R$ and $g^B:M\rightarrow\mathbb R$ as
\begin{equation*}    g(y)=m^1_y((f^{-1}\eta_1)(y))=m^1_y((f^{-1}\eta_2)(y))
\end{equation*}
\begin{equation*}\label{definition of average function}
    g^B(y)=\frac{1}{m^2_y(B^T)}\int_{B^T}m^1_z((f^{-1}\eta_2)(z))\,\td m^2_y(z)
\end{equation*}
where $\{m^1_y\}$ and $\{m^2_y\}$ are the conditional measures of $m$ associated with the measurable partitions $\eta_1$ and $\eta_2$. Here the second equality in the definition of $g$ holds because of \Cref{quotient structure}. 
Note that $g$ is a positive measurable function.

We now explain why $g^B$ is measurable for every uncentered ball $B$. 
It suffices to check measurability on $E$, since the case 
$y\in\bigcup_{n>0} f^nE\setminus E$ follows by pulling back to $E$. On $E$, note that the map $z\mapsto m_z^1(f^{-1}\eta_2(z))$ belongs to $L^1(m)$; indeed, one can choose a filtration $\{\cF_n\}_{n=1}^\infty$ of finite measurable partitions that increases to $\eta_2$, and use the martingale convergence theorem on $m_z^1(f^{-1}\cF_n(z))$. As $y\mapsto m^2_y$ is measurable, it follows that both $\displaystyle y\mapsto \int_{B^T}m^1_z((f^{-1}\eta_2)(z))\,\td m^2_y(z)$ and $y\mapsto {m_y^2(B^T)}$ are measurable,  hence so is $g^B$.
A similar argument shows that if $y\mapsto B_y = B(x(y),r(y))$ is measurable (in the sense that $y\mapsto (x(y),r(y))\in T\times \mathbb R^+$ is measurable, then 
$$
y\mapsto g^{B_y}(y)
$$ is measurable.

We now introduce the set of good (uncentered) balls. Intuitively,  an $(s+c)$-dimensional ball $B(x,r)$ that contains $y$ is called a good ball, if $\tau^{-1}y$ is close to the line $x+ E^c_{\omega_0}$ where $\omega_0$ is the density point that defines $T$, chosen in Section \ref{section 5.1}.

\begin{definition}[Good balls]\label{good ball def}
For any $y\in T$, the $(c+s)$-dimensional uncentered ball $B(x,r)\subseteq\tau^{-1}T$ is called a good ball w.r.t.\ $y$, if $\{a,b\}=\{\tau^{-1}y+ E^c_{\omega_0}\}\cap\partial B(x,r)$, then the angle $\wideparen{ab}$ between the vectors $a-x$ and $b-x$ lies in $(\pi-100\gamma_0,\pi] $.  The collection of good balls w.r.t.\ $y$ is denoted by $GB(y)$.
\end{definition}
For any $x\in E$ we have $\exp_x^{-1}\eta_2(x)\subseteq S^{cu}_\delta(x)\cap \exp_x^{-1}E\subseteq\tW^{cu}_{x,\delta}(x)$. 
We have proven in previous sections that $\pi^u(\exp_xS^{cu}_\delta(x)\cap E)\subseteq L^c_x$, $\tpi^u_x(\tW^{cu}_{x,\delta}(x))\subseteq L^c_x$, $\pi^u_x=\tpi^u_x := \pi^u\exp_x$ in $S^{cu}_\delta(x)\cap\exp_x^{-1}E$ and $\tau^{-1}L^c_x$ is a graph of a function $\varphi:R_{\omega_0}^c(\frac{1}{4}\delta_0A_0^{-1})\rightarrow R_{\omega_0}^{su}(\frac{1}{4}\delta_0A_0^{-1})$ with slope $\leq \frac{2\gamma_0}{C_1}$.
\begin{remark}
Good balls have the following property:  
for any good ball $B$ w.r.t.\ $y\in \tau^{-1}\pi^u(\eta_2(x))\subseteq\tau^{-1}L^c_x$, we have that  $\tau^{-1}L^c_x\cap B$ is a connected open interval inside the one dimensional $C^1$ manifold $\tau^{-1}L^c_x\cong \mathbb R$.
\end{remark}

The goal of Section \ref{section 5.3} is to establish the following lemma:
\begin{lemma}[Convergence arguments]\label{convergence arguments}
	The following statements hold:
	  
\noindent a) For $m$-a.e.\ $y$ one has
 $$\displaystyle\lim_{B\in GB(y),\,\diam(B)\to0} g^{B}(y)=g(y).$$\\
b) $\displaystyle \int-\log \bar{g}_*(y)\,\td m(y)<\infty$ where $\bar{g}_*(y) := \inf_{\{B\in GB(y)\}}g^B(y)$.
\end{lemma}

As a starting point, we have the following lemma.
\begin{lemma}\label{sub lemma of convergence}
    For $m$-a.e.\ $x\in M$ and $m^2_x$-a.e.\ $y$ we have \\
    a) $$\lim_{B\in GB(y),\,\diam(B)\to0}g^{B}(y)=g(y).$$\\ 
    b) $\displaystyle \int-\log \bar{g}_*(y)\,\td m^2_x(y)\leq H_{m_x^2}(f^{-1}\eta_2)+\log 2+1$ where $\bar{g}_*(y) := \inf_{\{B\in GB(y)\}}g^B(y)$.
\end{lemma}
We omit the full proof, giving only the main idea.
Note that as $\diam B\to 0$, the connected open intervals $\tau^{-1}L^c_x\cap B$, which can be regarded as one-dimensional uncentered balls of $y$, have their diameters going to zero. Then we can apply \Cref{1 dim convergence lemma} to the system $(\eta_2(x),m_x^2,\tau^{-1}L^c_x)$ with the projection map $\tau^{-1}\pi^u:\eta_2(x)\rightarrow \tau^{-1}L^c_x$ to obtain the lemma for $m$-a.e.\ $x\in E$.
When $x\in \bigcup_{n\geq 0}f^nE-E$, just note that the transverse metric and good balls are both defined using the previous return to $E$. 

Next, we show that the minimal function  $\bar{g}_*:M\rightarrow \mathbb R$ is measurable; as before, we only consider $y\in E$. We claim that
\begin{equation}\label{equation 54}
    \bar{g}_*(y)=\inf_{\{B = B(x,r)\in GB(y)\}}g^B(y)=\inf_{\{B = B(x,r)\in GB(y),\ r\in\mathbb Q,\ x\in\mathbb Q^{c+s}\}}g^B(y).
\end{equation}
Indeed, any ball $B\subseteq \tau^{-1}T$ can be approximated by an increasing sequence of rational balls 
$B(x_k,r_k)$ with $x_k\in\mathbb Q^{c+s}$ and $r_k\in\mathbb Q$, satisfying 
$B(x_1,r_1)\subseteq B(x_2,r_2)\subseteq\cdots$ and $\bigcup_k B(x_k,r_k)=B$. By \Cref{definition of average function} and the outer regularity of $m_y^2$, we have  $g^{B(x_k,r_k)}(y)\rightarrow g^B(y)$ for $m$-a.e.\ $y\in E$ for which $B\in GB(y)$. Thus the two infima in \eqref{equation 54} agree, and since the right-hand side is an infimum of countably many measurable functions, $\bar g_*$ is measurable.
%
%
%
%

Having established the measurability of $g^B$ and $\bar{g}_*$, integrating both parts of \Cref{sub lemma of convergence} over $y\in M$ as in \cite[Section 5.B]{LEDRAPPIER_YOUNG_A} gives \Cref{convergence arguments}.

\subsection{The ergodic argument}\label{section 5.4}
In this section we complete the proof of the Main Theorem. 
Recall that $\exp_x^{-1}\eta_2(x)\subseteq S^{cu}_\delta(x)\subseteq\tW^{cu}_{x,\delta}(x)=\tW^{cu}_x(x)\cap R_x(\delta l(x)^{-1})$. $\tW^{cu}_{x,\delta}(x)$ is sub-foliated by $\tW^u_{x,\delta}(y)$, $y\in \tW^u_{x,\delta}(x)$, and $\exp^{-1}_x(\eta_2(x)/\eta_1)$ in the $x$-chart coincides with the $\tW^u_x(y)$ partitioning on $\exp_x^{-1}\eta_2(x)$ as shown in \Cref{lemma 28}. 
Recall also that for any $x\in E$, $\tW^c_x(x)$ is a one-dimensional graph of a function $\varphi: E^c_x\rightarrow E^{us}_x$.

\begin{figure}[h!]
	\centering
	\def\svgwidth{\columnwidth}
	\includegraphics{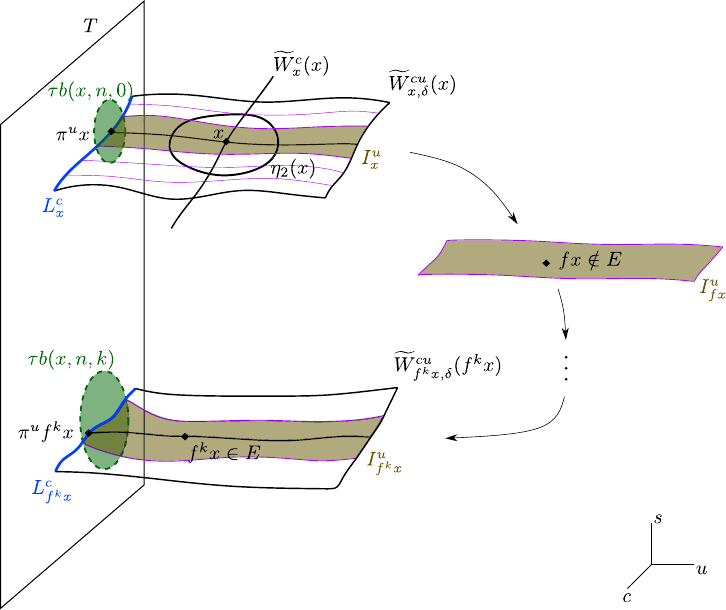}
	\caption{The construction of good balls $b(x,n,k)$.}
	\label{f.1}
\end{figure}

For $n\in \mathbb N$, we construct good balls $b(x,n,k)$ for $0\leq k\leq  [(1-\epsilon)n] :=  p$. When $k=0$,  the fake $u$-holonomy $\exp_x^{-1}\tpi^u_x: \tW^c_{x,\delta}(x)\rightarrow \exp_x^{-1}L^c_x$ is uniformly bi-Hölder as in \Cref{uniform Hölder}. We define 
\begin{equation}\label{e.rn}
r^n = 100^{-1}C_2^{-1}C_5^{-1}(\delta A_0^{-1}\mathrm{e}^{-3\epsilon n})^{\alpha^{-1}},
\end{equation}
\begin{equation*}
	b(x,n,0) :=  B(\tau^{-1}\pi^ux,r^n)\, \cap\,  R_{\omega_0}^{cs}(\delta_0A^{-1}_0),
\end{equation*}
and
\begin{equation*}
	a(x,n,0) := \Bigg( \bigcup_{z\in \tau b(x,n,0)\cap \tilde S}\hat\xi(z)\Bigg)\cap \eta_2(x),
\end{equation*}
where
\begin{equation*}
	\tS :=  \bigcup_{\hat\xi(y)\cap E\neq\emptyset}\hat\xi(y).
\end{equation*}
By definition, $\tau b(x,n,0)\subseteq T$ is an open neighborhood of $\pi^ux$ in $T$, and the diameter of $\exp_x^{-1}\tau b(x,n,0)$ in the $x$-chart is at most $C_5^{-1}(\delta A_0^{-1}\mathrm{e}^{-3\epsilon n})^{\alpha^{-1}}$.

Note that $\exp^{-1}_x\tau b(x,n,0)\cap\tW^{cu}_x(x)$ is an open interval in $\exp_x^{-1}L^c_x$ containing $x$, and when $n$ is large enough, we have $\exp^{-1}_x\tau b(x,n,0)\cap\tW^{cu}_x(x)\subseteq\exp_x^{-1}\tpi_x^u\tW^c_{x,\delta}(x)$.  We define
\begin{equation}\label{choice of Ix}
	I_x = (\exp_x^{-1}\tpi^u_x)^{-1}(\exp^{-1}_x\tau b(x,n,0))\cap\tW^{cu}_x(x) 
\end{equation}
Then we have $I_x\subseteq \tW^c_{x,\delta}(x)$ and by \Cref{uniform Hölder},
\begin{equation*}
	\mathrm{diam}(I_x)\leq \delta A_0^{-1}\mathrm{e}^{-3\epsilon n}.
\end{equation*}
We let $I^u_x :=  \bigcup_{z\in I_x}\{\tW^u_x(z)\}$, then $I^u_x\cap \tW^c_x(x)=I_x$. Since $\pi^u_x=\tpi^u_x$ in $\eta_2(x)$, we have $\exp_x^{-1}a(x,n,0)\subseteq I^u_x$.

Next we define $b(x,n,k)$ for $k>0$ using iterations of $I_x$. For this purpose, let $I_{f^nx} :=  \tf_x^n I_x$ and $I_{f^nx}^u :=  \tf_x^n I_x^u$, $n>0$. By the invariance of fake foliations, we have $I_{f^nx}=I^u_{f^nx}\cap \tW^c_{f^nx}(f^nx)$. The local argument (on center fake leaves) yields $|I_{f^kx}|\leq r_0\mathrm{e}^{-3\epsilon(n-k)}$ for $0\leq k\leq p$. Later we will send $n$ to infinity, along which $|I_{f^kx}|\to0$ uniformly in $x$.

For $0<k\leq p$ such that $f^kx\in E$, consider 
$$
\tilde I_{f^kx} :=I^u_{f^kx}\cap\exp_{f^kx}^{-1}T\cap R_{f^kx}\left(\frac{1}{4}\delta_0A_0^{-1}\right)\subseteq\exp_{f^kx}^{-1}L^c_{f^kx}.
$$ 
The inclusion holds when $n$ is large enough because of \Cref{uniform Hölder}.  which is a one-dimension open interval in $\exp_{f^kx}^{-1}L^c_{f^kx}$ containing $\exp_{f^kx}^{-1}\pi^u(f^kx)$. By \Cref{uniform Hölder} we have
\begin{equation*}
    \mathrm{diam}\left(\tilde I_{f^kx}\right)\leq C_2r_0^\alpha \mathrm{e}^{-3\epsilon\alpha(n-k)}
\end{equation*}
in the $f^kx$-chart. Hence
\begin{equation*}
      \mathrm{diam}\left(\tau^{-1}\exp_{f^kx}\tilde I_{f^kx}\right)\leq 2C_5C_2r_0^\alpha \mathrm{e}^{-3\epsilon\alpha(n-k)}
\end{equation*}
inside $\tau^{-1}T\subseteq \mathbb R^{c+s}$ with the Euclidean metric.

Note that $\tau^{-1}\exp_{f^kx}\tilde I_{f^kx}$ is an open arc in $\tau^{-1}L^c_{f^kx}$, and its endpoints determine a line segment that we use as a diameter to construct a ball $B\subseteq \tau^{-1}T$ with radius $r^n_k\leq C_5C_2r_0^\alpha \mathrm{e}^{-3\epsilon\alpha(n-k)}$. Now we define
\begin{equation*}
    b(x,n,k) :=B,
\end{equation*}
and
\begin{equation*}
    a(x,n,k) := \Bigg( \bigcup_{z\in \tau b(x,n,k)\cap \tilde S}\hat\xi(z)\Bigg)\cap \eta_2(f^kx).
\end{equation*}
See Figure \ref{f.1}.
We claim that $b(x,n,k)$ is a good ball w.r.t.\ $\tau^{-1}\pi^u f^kx$. 
This is because by \Cref{section 5.1}, $\tau^{-1}L^c_{f^{k}x}$ is a graph of a function $\varphi:R^c_{\omega_0}(\frac{1}{4}\delta_0A_0^{-1})\rightarrow R_{\omega_0}^{s}(\frac{1}{4}\delta_0A_0^{-1})$ whose slope is at most $\frac{2\gamma_0}{C_1}$.

Finally, if $f^kx\notin E$, we choose the maximal $k_1<k$ such that $f^{k_1}x\in E$. Then we let
\begin{equation*}
    a(x,n,k) :=  f^{k-k_1}a(x,n,k_1)\cap \eta_2(f^kx),
\end{equation*}
\begin{equation*}
    b(x,n,k) :=  b(x,n,k_1).
\end{equation*}
Hence $r^n_k:=\diam b(x,n,k) \leq C_5C_2r_0^\alpha \mathrm{e}^{-3\epsilon\alpha(n-k)}$ still holds, and $b(x,n,k)$ is a good ball for every $0\leq k\leq p$.

Our construction ensures the following property.
\begin{lemma} For every $0\le k < p$ one has
    $f(a(x,n,k))\cap\eta_2(f^{k+1}x)\subseteq a(x,n,k+1).$
\begin{proof}
Recall that $\eta_2(f^{k+1}x)\subseteq S_\delta^{cu}(f^{k+1}x)\subseteq \tW^{cu}_{f^{k+1}x,\delta}(f^{k+1}x)$. By \Cref{lemma 8 the size}, we have $W^u_{x,2\delta}(y)=\tW^u_{x,2\delta}(y)$ and $\pi^u_{f^{k+1}x}=\tpi^u_{f^{k+1}x}$ inside $\exp_{f^{k+1}x}^{-1}\eta_2(f^{k+1}x)$. We also have $\tf_xI^u_{f^kx}=I^u_{f^{k+1}x}$.

Then by \Cref{first quotient structure} and \Cref{quotient structure}, and the definition of $a(x,n,k)$, we have that $a(x,n,k)=\{y\in\eta_2(f^kx):\exp^{-1}_{f^kx}\tpi^u_{f^kx}\exp_{f^kx}^{-1}y\in I_{f^kx}\}$. Thus we obtain 
\begin{equation*}
a(x,n,k)=\exp_{f^kx}(I^u_{f^kx}\cap\exp^{-1}_{f^kx} \eta_2(f^kx)).    
\end{equation*}
The lemma follows by noting that $\tf_xI^u_{f^kx}=I^u_{f^{k+1}x}$ and $f(\eta_2(f^kx))\supseteq\eta_2(f^{k+1}x)$.
\end{proof}
\end{lemma}

Below we carry out a calculation  similar to \cite[Section 5.D]{LEDRAPPIER_YOUNG_A}. First we have
\begin{equation*}
\begin{aligned}
    m^2_xa(x,n,0)
                 &\leq \prod_{k=0}^{p-1}\frac{m^2_{f^kx}a(x,n,k)}{m^2_{f^{k+1}x}a(x,n,k+1)}\\
                 &= \prod_{k=0}^{p-1}m^2_{f^kx}a(x,n,k)\cdot\frac{m^2_{f^kx}f^{-1}(\eta_2(f^{k+1}x))}{m^2_{f^{k}x}f^{-1}(a(x,n,k+1))}\\
                 &\leq \prod_{k=0}^{p-1}\frac{m^2_{f^kx}a(x,n,k)}{m^2_{f^kx}(f^{-1}\eta_2(f^kx)\cap a(x,n,k))}\cdot m^2_{f^kx}(f^{-1}\eta_2(f^{k}x))\\
                 &=\prod_{k=0}^{p-1}g^{b(x,n,k)}(f^kx)^{-1}\cdot \mathrm{e}^{-I_2(f^kx)}\
\end{aligned}
\end{equation*}
where $I_2(x) :=  -\log m^2_x(f^{-1}\eta_2(x))$ is the information function of $(f^{-1}\eta_2|\eta_2)$.

Taking logarithms, dividing by $n$, applying \Cref{dimension lemma} and keeping in mind Equation \eqref{e.rn}, we obtain
\begin{equation}\label{equation dimension limit}
\begin{aligned}
    \frac{3\epsilon}{\alpha}(s+1)&\geq \liminf_{n\to\infty} \frac{|\log m^2_xa(x,n,0)|}{n}\\
                                 &\geq \liminf_{n\to\infty}\frac{1}{n}\left(\sum_{k=0}^{p-1}\log g^{b(x,n,k)}(f^kx)+\sum_{k=0}^{p-1}I_2(f^kx)\right).   
\end{aligned}
\end{equation}

Since $b(x,n,k)$ are good balls and $k\le p = (1-\epsilon)n$, we have
\begin{equation*}
    r^n_k\leq  C_5C_2r_0^\alpha \mathrm{e}^{-3\epsilon\alpha(n-k)}\leq  C_5C_2r_0^\alpha \mathrm{e}^{-3\epsilon^2\alpha n}
\end{equation*} which tends to 0 as $n\to\infty$, uniformly in $x$ and $k$. We now apply the convergence arguments to complete the proof of the Main Theorem as in \cite[Section 5.E]{LEDRAPPIER_YOUNG_A}.

By \Cref{convergence arguments}, there exists a measurable function $r(x)>0$ such that for $m$-a.e.\ $x$, if $B\in GB(x)$ and $\diam(B)<r(x)$, then $-\log g^B(x)\leq-\log g(x)+\epsilon$. Moreover, since $\int-\log \bar{g}_*\,\td m<\infty$, there exists $r_*>0$ such that $A := \{x:r(x)\geq r_*\}$ satisfies $\int_{M-A}-\log \bar{g}_*\,\td m\leq\epsilon$.

Take $n$ large enough so that $r^n_k<r_*$ for $0\leq k\leq p$. Then
\begin{equation*}
\begin{aligned}
    \sum_{k=0}^{p-1}-\log g^{b(x,n,k)}(f^kx)&=\sum_{f^kx\in A}-\log g^{b(x,n,k)}(f^kx)+\sum_{f^kx\notin A}-\log g^{b(x,n,k)}(f^kx)\\
    &\leq \sum_{f^kx\in A}(-\log g(f^kx)+\epsilon)+\sum_{f^kx\notin A}-\log \bar{g}_*(f^kx).
\end{aligned}
\end{equation*}
By the Birkhoff ergodic theorem, for $m$-a.e.\ $x\in E$ it holds that
\begin{equation*}
\begin{aligned}
    \limsup_{n\rightarrow\infty}\frac{1}{n}\sum_{k=0}^{p-1}-\log g^{b(x,n,k)}(f^kx) &\leq 
    (1-\epsilon)\left(\int-\log g(z)\,\td m(z)+\epsilon+\epsilon\right).\\
    &=(1-\epsilon)(h_m(f,\eta_1)+2\epsilon).
\end{aligned}
\end{equation*}
Consequently,
\begin{equation*}
\begin{aligned}
\liminf_{n\rightarrow\infty}&\frac{1}{n}\left(\sum_{k=0}^{p-1}\log g^{b(x,n,k)}(f^kx)+\sum_{k=0}^{p-1}I_2(f^kx)\right)\\
\geq&-(1-\epsilon)(h_m(f,\eta_1)+2\epsilon)+(1-\epsilon)h_m(f,\eta_2)\\
\geq&-(1-\epsilon)(h^u_m(f)+2\epsilon)+(1-\epsilon)(h_m(f)-\epsilon).
\end{aligned}
\end{equation*}
Hence Equation \ref{equation dimension limit} becomes 
\begin{equation}\label{final property}
     \frac{3\epsilon}{\alpha}(s+1)\geq-(1-\epsilon)(h^u_m(f)+2\epsilon)+(1-\epsilon)(h_m(f)-\epsilon).
\end{equation}
Letting $\epsilon\rightarrow0$, we conclude that $h^u_m(f)\geq h_m(f)$. Since $h^u_m(f)\leq h_m(f)$ holds trivially, it follows that $h^u_m(f)=h_m(f)$, completing the proof of the Main Theorem.
\begin{remark}
    We make the following observations:\\
    a) The H\"older index $\alpha=1-\frac{7\epsilon}{\lambda^u}\rightarrow 1$ as $\epsilon\rightarrow0$.\\
    b) The term $s+c=s+1$ on the left-hand side of \Cref{equation dimension limit} can be replaced by $1$ because the quotient measure of $m^2_x$ is supported on the one-dimensional manifold $L^c_x$.\\
    These refinements are not needed for Part 1, as the left-hand side of \Cref{final property} tends to zero regardless; however, they will be important for the second part of this project.
\end{remark}

\section{Invariance principle and applications}\label{section6}
In this section, we present several applications of the Main Theorem by generalizing the invariance principle to $C^1$ diffeomorphisms with dominated splitting and one-dimensional center. 
\subsection{Invariance principle} Following  \cite{UVY}, a diffeomorphism $f$ is called partially hyperbolic  if the tangent bundle has a dominated splitting $ TM= E^u \oplus  E^c \oplus  E^s $ such that $ E^s $ is uniformly contracting and $ E^u $ is uniformly expanding. A partially hyperbolic diffeomorphism $f$ is dynamically coherent, if for $i=cs,cu$ there are invariant foliations $\cF^i$ tangent to $E^i$, where $E^{cs}=E^c\oplus E^s$ and $E^{cu}=E^c\oplus E^u$. $f$ is accessible if any two points $x,y\in M$ can be joined by a curve formed by finitely many arcs that are contained in leaves of $ \cF^u $ or $ \cF^s$.

\begin{definition}[Special partially hyperbolic diffeomorphisms]
	Denote by $\mathrm{SPH}(M)$ the set of partially hyperbolic, accessible, dynamically coherent $C^1$ diffeomorphisms with one-dimensional center direction, such that the center foliation $\cF^c$ forms a circle bundle and the quotient space $M_c:=M/\cF^c$ is a torus.
\end{definition}

Below we always assume $f\in \mathrm{SPH}(M)$. Let $f_c$ denote the map induced by $f$ on $M_c$, then $f_c$ is a topological Anosov homeomorphism in the sense of \cite[Section 1.3]{Viana08} or  \cite[Section 2.2]{VY13}.
By a result of Hiraide in \cite{hiraide90}, $f_c$ is conjugate to an Anosov automorphism on the torus. 
In particular $f_c$ is transitive and has a unique probability measure of maximal entropy, which we denote by $\nu$.

\begin{center}
	\begin{tikzcd}
		{( M, m)} \arrow[rrr, "f"] \arrow[dd, "\pi_c"] &  &  & {( M,m)} \arrow[dd, "\pi_c"] \\
		&  &  &                                  \\
		{(M/\cF^c,\hat m:=(\pi_c)_* m)} \arrow[rrr, "f_c"]                    &  &  & {(M/\cF^c,\hat m:=(\pi_c)_* m)}                    
	\end{tikzcd}
\end{center}

An invariant measure $m$ is called $u$-invariant if its conditional measures on the center fibers $\{m^c\}$ are invariant under the unstable holonomies for $m$-a.e.\ $x\in M$. 
With our Main Theorem, we can now obtain a $C^1$ version of the invariance principle in \cite[Corollary 2.1]{TY} directly from \cite[Theorem A]{TY}

\begin{theorem}\label{thm invprin}
	Let $f\in \mathrm{SPH}(M)$ and $m$ be an $f$ invariant probability measure. If the central Lyapunov exponent of $m$ is non-positive almost everywhere, then $m$ is $u$-invariant.
\end{theorem}

\begin{proof}
	We sketch the proof using the following diagram.

	\begin{center}
		\begin{tikzcd}
			{(M, h_m(f))} \arrow[rrr, "{4):\ \lambda^c\leq 0,\, h^u_m(f)=h_m(f)}"] \arrow[dd, "1):\ h_m(f)=h_{\hat m}(f_c)"] &  &  & {(M,h^u_m(f))} \arrow[dd, "3): \ h^u_m(f)\leq h^u_{\hat m}(f_c)"] \\
			&  &  &                                                                                                          \\
			{(M/\cF^c,h_{\hat m}(f_c))} \arrow[rrr, "2):\ h_{\hat m}(f_c)=h^u_{\hat m}(f_c)"]                                    &  &  & {(M/\cF^c,h_{\hat m}^u(f_c))}                                                                                 
		\end{tikzcd}
	\end{center}
	
	1) The Ledrappier-Walters variational principle (\cite{LedrappierWalters}) gives
	$$\sup_{(\pi_c)_*m=\hat m} h_m(f)=h_{\hat m}(f_c)+\int_{M/\thF^c}h(f,\pi_c^{-1}(y))\,\td\hat m(y).$$
	Since $M$ is a circle bundle, we have $h(f,\pi^{-1}_c(y))=0$ and hence  $h_m(f)=h_{\hat m}(f_c)$. 
	
	2) Since $f_c$ conjugates to an Anosov automorphism on torus, whose unstable entropy is equal to the metric entropy, we have $h_{\hat m}(f_c)=h^u_{\hat m}(f_c)$.
	
	3) \cite[Theorem A]{TY} shows that for $f\in \mathrm{SPH}(M)$, $h^u_m(f)\leq h^u_{\hat m}(f_c)$ with equality if and only if $m$ is $u$-invariant. This result was proven using Jensen's inequality and does not rely on regularity higher than $C^1$. Note that the inequality direction is opposite to that in Step 1.
	
	4) For $f\in\mathrm{SPH}(M)$, if the center exponent $\lambda^c\leq 0$, then the main theorem gives $h^u_m(f)=h_m(f)$. Combining with 1), 2) and 3) we see that the equality in 3) must hold and therefore $m$ is $u$-invariant due to \cite[Theorem A]{TY}.  
\end{proof}

Crovisier and Poletti \cite{crovisier2023invarianceprinciplenoncompactcenter} recently proved an invariance principle for some partially hyperbolic diffeomorphisms with non-compact center leaves and obtained  numerous applications. Combining their Main Theorem (which is stated for $C^1$ diffeomorphisms) and our Main Theorem, one can directly obtain a $C^1$ version of \cite[Corollary D]{crovisier2023invarianceprinciplenoncompactcenter}.
\begin{theorem}
	Let $f$ be a partially hyperbolic diffeomorphism with one-dimensional quasi-isometric center (see \cite{crovisier2023invarianceprinciplenoncompactcenter} for precise definition) and let $m$ be an ergodic measure. If the center Lyapunov exponent of $m$ is non-positive, then the center disintegration $\{m^c_x\}$ is $u$-invariant.
\end{theorem}

\subsection{Applications in $\mathrm{SPH}(M)$}
In this section, we present several results on measures of maximal entropy (and on measures with sufficiently large entropy) 
for diffeomorphisms in $\mathrm{SPH}(M)$. Such results have been obtained in \cite{HHTU}, \cite{TY}, \cite{UVY}, \cite{UVYY2024} and \cite{UVYY2025} for $C^2$ systems.

We first give the definitions of rotation type and hyperbolic type:
\begin{definition}[See \cite{TY}]
	$f\in\mathrm{SPH}(M)$ is of rotation type if, after continuously changing the metrics on the center fibers, the restriction of $f$ to each center fiber is a circle rotation. If $f$ is not of rotation type, $f$ is said to be of hyperbolic type.
\end{definition}

The following proposition, first proven in \cite[Theorem 1]{HHTU} for $C^2$ diffeomorphisms (see also \cite{UVY} and \cite{UVYY2024}), classifies the measures of maximal entropy for $\mathrm{SPH}(M)$.
See \cite[Section 1]{UVYY2024} for the definition of $c$-mostly contracting center, which is defined for $C^1$ diffeomorphisms.
\begin{proposition}\label{proposition 6}
	Let $f\in\mathrm{SPH}(M)$. If $f$ is of rotation type, then $f$ admits a unique measure of maximal entropy with vanishing center exponent. Moreover, $f$ has no hyperbolic periodic orbit.

	If $f$ is of hyperbolic type, then $f$ has only finitely many ergodic measures of maximal entropy, all of which are hyperbolic; some of them have positive center exponent, and the others have negative center exponent. Furthermore, $f$ has $c$-mostly contracting center and admits hyperbolic periodic orbits.
\end{proposition}

We postpone the proof of Proposition \ref{proposition 6} to the next subsection.

\begin{remark}
	By \Cref{proposition 6} we have that $f\in\mathrm{SPH}(M)$ is of rotation type if and only if it has no hyperbolic periodic orbit, if and only if it has only one measure of maximal entropy, if and only if it admits a single non-hyperbolic ergodic measure of maximal entropy. 
\end{remark}

\begin{remark}
	When the system is of rotation type and $C^2$, it was shown in \cite{HHTU} that the condition measures $\{m^c_x\}$ of the non-hyperbolic ergodic measure of maximal entropy $m$ are equivalent to the Lebesgue measures on the center circle fibers. This is the key step in the proof and heavily relies on the $u/s$ holonomy maps inside $cu/cs$ leaves being Lipschitz, which is not true in general for $C^1$ diffeomorphisms. We do not know whether the conditional measures $m^c_x$ are equivalent to the Lebesgue measure on the center fibers when the system is only $C^1$. 
\end{remark}

We then introduce the following result concerning the rigidity of high entropy measures.

\begin{proposition}\label{proposition 7}
	Suppose $f\in\mathrm{SPH}(M)$ is of hyperbolic type. Then there exist $\vep>0$ and
	$\lambda_0>0$ such that for every ergodic invariant probability measure of $f$ whose entropy exceeds $h_{top}(f)-\vep$, its center exponent satisfies $|\lambda^c(m)|>\lambda_0$.
\end{proposition}
This result was proven in \cite[Theorem B]{TY} assuming $C^2$ regularity. However, the $C^2$ assumption is only used to ensure the invariance principle and the classification of rotation type in \cite[Theorem 1]{HHTU}. We have proven in \Cref{main thm} and \Cref{proposition 6} that they both hold for $C^1$ diffeomorphisms under this setting, and thus \Cref{proposition 7} remains true.

The next two results follow directly from \Cref{proposition 6} and the results in \cite[Theorem C]{UVYY2024} and \cite[Theorem C]{UVYY2025}, all of which only require $C^1$ smoothness. See \cite{UVYY2024} for the precise definition of skeletons.

\begin{proposition}\label{prop8}
	Suppose $f$ is of hyperbolic type, and denote $m_1,\cdots, m_k$ to be all the ergodic measures of maximal entropy with negative center exponent. Then the supports of $m_i$ are $u$-saturated, pairwise disjoint, and each support has finitely many connected components. Furthermore, the unstable foliation inside each connected component is minimal. Moreover, these measures can be topologically classified by skeletons.
\end{proposition}

\begin{proposition}
	Suppose $f\in \mathrm{SPH}(M)$ is of hyperbolic type, and $m$ is one of its ergodic measures of maximal entropy with negative center exponent. Suppose $\supp(m)=\bigcup_{i=1}^\ell \Lambda_i$ has $\ell$ connected components. Denote by $m_1$ the ergodic decomposition of $m$ with respect to $f^\ell$ supported on $\Lambda_1$. Then $(f^\ell,m_1)$  has exponential decay of correlations for Hölder continuous functions.
\end{proposition}
Using \Cref{prop8} and the proof in \cite{UVY}, we can immediately obtain a $C^1$ version of \cite[Theorem D]{UVY}, characterizing the variation of the numbers of measures of maximal entropy under $C^1$ perturbations of $f$. Write $\tP(M)$ for the space of probability measures on $M$ endowed with the $\text{weak}^*$ topology.
\begin{proposition}
	If $f\in\mathrm{SPH}(M)$ is of hyperbolic type, then there is a $C^1$-neighborhood $\mathcal U$ of $f$ such that the number of ergodic measures of maximal entropy for any $g\in\mathcal U$ with negative (resp. positive) center exponent is smaller than or equal to the number of measures of maximal entropy for $f$ with negative (resp. positive) center exponent.

	If $f$ is of rotation type, then there is a $C^1$-neighborhood $\mathcal U$ of $f$ such that any $g\in\mathcal U$ has at most two ergodic measures of maximal entropy. Moreover, there exists two continuous functions $\Gamma^+$ and $\Gamma^-:\mathcal U\rightarrow \tP(M)$ such that $\Gamma^-(g)$ (resp. $\Gamma^+(g)$) is an ergodic measures of maximal entropy of $g$ with non-positive (resp. non-negative) center exponent.
\end{proposition}

Finally, we remark that the classification of partially hyperbolic diffeomorphisms on 3-Nilmanifolds other than $\mathbb T^3$ was provided by Hammerlindl and Potrie in \cite{HP14}. In this case, $f$ is always dynamically coherent with center leaves forming a circle bundle. Moreover, $f$ is always accessible, and supports a unique invariant $u$-saturated set which is connected. Thus, following the previous results, the following result in \cite{UVY} also holds for $C^1$ diffeomorphisms.
\begin{proposition}
	Let $f$ be a partially hyperbolic diffeomorphism on $\mathrm{Nil}^3\setminus \mathbb T^3$, and assume that $f$ has a hyperbolic periodic point. Then $f$ admits two ergodic measures of maximal entropy $m^+$ and $m^-$ with center exponent positive and negative respectively. Both measures have exponential decay of correlations for Hölder observables. Furthermore, all ergodic measures with sufficiently large entropy are hyperbolic.
\end{proposition}

\subsection{Proof of \Cref{proposition 6}}\label{Section 6}
By \Cref{thm invprin}, the invariance principle holds in our situation, and the proof in the case when $f$ admits no ergodic non-hyperbolic measure of maximal entropy is the same as in \cite{HHTU} and \cite{UVYY2024}. Thus, we only focus on the proof for the alternative case. 

Suppose $f$ admits an ergodic, non-hyperbolic measure $m$ of maximal entropy.  According to the Ledrappier-Walters variational principle in \cite{LedrappierWalters},
$$\sup_{(\pi_c)_*m=\hat m} h_m(f)=h_{\hat m}(f_c)+\int_{M/\thF^c}h(f,\pi_c^{-1}(y))\,\td\hat m(y),$$
so $(\pi_c)_* m = \nu$ is the measure of maximal entropy on the quotient space. By Hiraide \cite{hiraide90}, $f_c$ is conjugate to a linear Anosov torus diffeomorphism. Thus, $\nu$ has local product structure and is fully supported. 

Since the center Lyapunov exponent of $m$ is $0$, the invariance principle (\Cref{thm invprin}) can be applied to both $f$ and $f^{-1}$ to obtain that the conditional measures $m^c_x$ along the center leaves are invariant under both $s$ and $u$ holonomy maps and are invariant under iteration by $f$. Moreover, using a classical Hopf argument as in the proof of \cite[Theorem D]{AV}, the family $\{m^c_x\}$ can be chosen to depend continuously on the base point $x$.

Next, we show that $m^c_x$ has no atoms and $\supp(m^c_x)=\cF^c(x)$. 

The proof follows directly from $s/u$ invariance.
Suppose first that $m^c_x$ contains an atom $y$ with $m^c_x(y)=a>0$. Since $f$ is accessible, for any point $z\in \cF^c(x)$, there is a map $H:\cF^c(x)\to \cF^c(x)$ which is the composition of finitely many $s/u$ holonomy maps, such that $H(y)=z$. Because the measure $m^c_x$ is invariant under $H$, we have that $m^c_x(z)=a>0$ for all $z$, which is a contradiction.
A similar proof shows that if an interval $I\subseteq\cF^c(x)$ satisfies $m^c_x(I)= 0$, then for any $z\in\cF^c(x)$ there is an interval $I_z$ such that $m^c_x(I_z)=0$. This would imply $\supp(m^c_x)=\emptyset$ which is impossible. Hence $\supp(m^c_x)=\cF^c(x)$.

Next we define a continuous change of metrics on the center fibers. Consider the following metric $d^c$ on the center leaf $\cF^c(x)$: for $y,z\in \cF^c(x)$, $d^c(y,z):=\min\{m^c_x([y,z]),m^c_x([z,y])\}$ where $[y,z]$ is the anticlockwise arc from $y$ to $z$. Because $m^c_x$ has no atoms and is fully supported, $d^c$ is a metric on the center leaf. Moreover, because the conditional measures are continuous w.r.t. the base point, the metrics change continuously with respect to the center leaves.
Finally, since the conditional measures are $f$-invariant, $f$ is an isometry on the center leaves under the new metric. This shows that $f$ is of rotation type. In particular, $f$ admits no hyperbolic periodic orbit.

It remains to show that $m$ is the unique measure of maximal entropy for $f$. Suppose on the contrary that $f$ admits another ergodic measure $\omega$. Because $f$ admits no hyperbolic periodic orbit, $\omega$ must be non-hyperbolic. This follows from an argument using Liao's shadowing lemma as in \cite{Gan2002}: if $\omega$ is hyperbolic then there exists a hyperbolic periodic orbit; see \cite[Lemma 4.5]{UVY} for example. By the invariance principle (\Cref{thm invprin}), $\omega$ admits a continuous family of conditional measures along the center leaves $\{\omega^c_x\}$ which are $s/u$-invariant and $f$-invariant.

Because $m$ and $\omega$ are distinct ergodic measures with the same quotient measure $\nu$ on the quotient space of center leaves, their conditional measures are mutually singular on typical center leaves. For such a typical leaf $\cF^c(y)$, there is a point $z\in \cF^c(y)$ and an open interval neighborhood $I_z$ of $z$ such that
$\frac{m^c_y(I_z)}{\omega^c_y(I_z)}>100$. By the accessibility of $f$ and compactness of $\cF(y)$, there are finitely many maps $H_i: \cF^c(y)\to \cF^c(y)$, all of which are finite compositions of $s/u$ holonomy maps, such that the images
$I_i=H_i(I_z)$ cover the entire leaf $\cF^c(y)$. By taking a subcover and applying \Cref{1 dim convergence lemma}, we may assume that each point of the center leaf is covered by at most two $I_i$'s. This means that $\sum_i m^c_y(I_i)\leq 2$.
By the $s/u$-invariance of $m^c_y$ and $\omega^c_y$, we know that $\frac{m^c_y(I_i)}{\omega^c_y(I_i)}=\frac{m^c_y(I_z)}{\omega^c_y(I_z)}>100$, which leads to $\omega^c_y(\cF^c(y))\leq \sum_i \omega^c_y(I_i) \leq \sum_i m^c_y(I_i)/100\leq 2/100<1$. This contradicts the fact that $\omega^c_y(\cF^c(y))=1$. Hence $f$ admits a unique measure of maximal entropy. This finishes the proof of \Cref{proposition 6}. \qed

\section*{Acknowledgements}
The authors are grateful to Fan Yang for many helpful conversations and comments on a first version of the text, and for his help to rewrite certain parts of this paper.

Shaobo Gan is partially supported by National Key R\&D Program of China 2022YFA1005801, and Yao Tong is partially supported by FRPUS. Jiagang Yang was partially supported by CAPES-Finance Code 001, CNPq-Brazil grant 312054/2023-8, CNPq-Projeto Universal No.
404943/2023-3, PRONEX, NSFC 12271538, NSFC 11871487, NSFC 12071202 and
MATH-AmSud 220029.
\bigskip

\noindent {\bf Data Availability}
No numerical or categorical data was used in this article.

\noindent {\bf Declarations}

\noindent {\bf Conflict of interest} 
The authors have no relevant financial or non-financial interests to disclose.

\bibliographystyle{abbrv}
{\footnotesize\bibliography{library}}

\end{document}